\documentclass[a4paper, 11pt]{article}
\usepackage{amsmath,amsfonts,amssymb,amscd,amsthm,amsbsy}

\usepackage{url}
\usepackage{soul}
\usepackage{amsthm}
\usepackage{authblk}
\usepackage{amssymb}
\usepackage{amsfonts}
\usepackage{graphicx}
\usepackage{verbatim}
\usepackage{mathrsfs}
\usepackage{subfigure}
\usepackage{fancyhdr}
\usepackage{latexsym}
\usepackage{graphicx}
\usepackage{dsfont}
\usepackage{color}
\definecolor{shadecolor}{rgb}{1, 0.8, 0.3}
\usepackage{algorithm}
\usepackage{algpseudocode}
\usepackage{rotating}
\usepackage {blindtext}
\usepackage{geometry}
\usepackage{fancyhdr}
\usepackage{amsmath , amsthm , amssymb}
\usepackage{amsfonts}
\usepackage{graphicx}
\usepackage{enumitem}
\normalfont
\usepackage[T1]{fontenc}
\usepackage{relsize}
\usepackage[utf8]{inputenc}
\usepackage[english]{babel}
\usepackage{fancyhdr}
\usepackage{enumitem,graphicx,appendix}
\usepackage{amsmath,amssymb,amsthm}
\usepackage{xargs}
\usepackage{marginnote,color}

\usepackage{hyperref}
\hypersetup{
    colorlinks=true,
    linkcolor=blue,
    citecolor=blue,
    filecolor=magenta,      
    urlcolor=cyan
}

\theoremstyle{plain}
\newtheorem{theorem}{Theorem}[section]

\newtheorem{lemma}[theorem]{Lemma}

\newtheorem{remark}[theorem]{Remark}

\theoremstyle{definition}

\newcommand{\R}{\mathbb R}
\newcommand{\Z}{\mathbb Z}

\title{Weighted Birkhoff Averages \\
and the Parameterization Method 
}
\author{David Blessing \thanks{
Email: {\tt dblessing2014@fau.edu}} }
\author{J.D. Mireles James \thanks{
Email: {\tt jmirelesjames@fau.edu}}}
\affil{Florida Atlantic University, Department of Mathematical Sciences}

\begin{document}

\maketitle

\abstract{
This work provides a systematic recipe for  
computing accurate high order Fourier expansions of
quasiperiodic invariant circles (and systems of such circles)
in area preserving maps.  The recipe requires only a finite data set sampled
from the quasiperiodic circle.  Our approach, being based on 
the parameterization method of \cite{MR2299977,MR2240743,MR2289544},
uses a Newton scheme to iteratively solve a conjugacy equation describing
the invariant circle (or systems of circles).  A critical step in properly formulating 
the conjugacy equation is to determine the rotation number of the quasiperiodic 
subsystem.  For this we exploit the the weighted Birkhoff averaging method  
of \cite{MR3718733,MR3526529,MR3755876}.  This approach facilities 
accurate computation of the rotation number given nothing but 
the already mentioned orbit data.  

The weighted Birkhoff averages also facilitate the  
computation of other integral observables like 
Fourier coefficients of the parameterization of the invariant
circle.  Since the parameterization method is based on a Newton scheme, 
we only need to approximate a small number of Fourier coefficients 
with low accuracy (say, a few correct digits) to find a good enough initial
approximation so that Newton converges.  Moreover, the Fourier coefficients
may be computed independently, so we can sample the higher modes to guess 
the decay rate of the Fourier coefficients.  This allows us to choose, a-priori, 
an appropriate number of modes in the truncation.

We illustrate the utility of the approach for explicit example systems
including the area preserving Henon map 
and the standard map (polynomial and trigonometric nonlinearity respectively).
We present example computations for (systems of) invariant circles
with period as low as 1 and up to more than 100.  We also employ 
a numerical continuation scheme
(where the rotation number is the continuation parameter)
to compute large numbers of quasiperiodic circles in these systems.  
During the continuation we monitor the 
Sobolev norm of the Parameterization, as explained in \cite{MR2672635}, 
to automatically detect the breakdown of the family.  
}

%\tableofcontents

\section{Introduction}
\label{sec:intro}

%The present work is concerned with accurate computation of  
%invariant circles for area preserving, planar diffeomorphisms --
%especially in the case where the existence of the invariant circle
%is suggested by the results of numerical simulations
%or experimental data, rather than by perturbative considerations.  
%Our goal is to compute a Fourier series expansion of 
%the invariant circle using the Parameterization Method 
%of de la Llave and Haro \cite{MR2240743,MR2289544,MR2299977}.

Suppose that $\Gamma$ is an invariant torus of a discrete or continuous 
time dynamical system.  We say that $\Gamma$ is a rotational invariant 
torus if the dynamical on $\Gamma$ are are topologically conjugate to 
independent irrational rotations.  A quasiperiodic orbit is any orbit on 
a rotational invariant torus and, since the rotations are independent, 
all such orbits are dense in the torus.  

Cantor families of invariant tori are common in structure preserving 
dynamical systems like reversible maps, 
area and volume preserving maps on manifolds, 
and also for higher dimensional generalizations to symplectic maps
on (even dimensional) symplectic manifolds.
Indeed, for such systems typical orbits are observed to 
be either chaotic or quasiperiodic.  Given a long enough finite orbit 
segment sampled from an invariant torus, an important problem is
to be able to rapidly and accurately approximate a parameterization 
of the invariant torus.  

Two powerful approaches for solving this problem are given by the 
Parameterization method, and the method of exponentially weighted 
Birkhoff sums.   The Parameterization method is a functional 
analytic framework for studying invariant manifolds on which the 
dynamics are conjugate to a known simple model, and was 
developed in detail for invariant tori (and their stable/unstable manifolds)
in the three papers \cite{MR2289544,MR2240743,MR2299977}.
This approach is discussed in detail in 
Section \ref{sec:parmMethod}, where a number of additional references
are given.  At the moment we simply stress that the idea
of the parameterization method is to develop 
Newton schemes for solving the conjugacy equation describing 
the unknown parameterization of the invariant torus (or other invariant manifold).  

When working in a non-perturbative setting, 
two challenges are to (i) determine the
rotation number of the desired invariant circle, and (ii) to
 produce an accurate enough initial condition
so that the Newton method converges.  
Another important question is to choose an appropriate 
truncation dimension for the desired parameterization
(number of Fourier modes with which to compute).

The approach proposed here uses the weighted Birkhoff averages developed in 
\cite{MR3755876,MR3526529,MR3718733,MR4104977,MR4322369}
to efficiently obtain this information directly from data (a long enough orbit segment).
By combining the Parameterization Method with the 
weighted Birkhoff averages just mentioned, we obtain a general 
and non-perturbative procedure which allows us to compute the desired Fourier 
 expansion accurately and to high order.  
 Since the method is iterative, the coefficients can typically be 
 computed to machine precision.  
 Moreover, since the parameterization method is based on 
 solving a functional equation, it comes equipped with
  a natural notion of a-posteriori error.

 We remark that a great many previous studies deal with numerical methods 
 for computing invariant circles/tori in area preserving/symplectic maps and 
 Hamiltonian systems.  While a thorough review of the literature is beyond the 
 scope of the present work, we refer the interested reader to the papers of 
 \cite{MR2240743,MR1655359,
 MR2672635,MR2241302,MR2507323,MR1210989,
 MR1862806,MR3713932,MR2967458,MR4263036,
 MR3477312,MR4344967,MR3388608}
and the references cited therein.  
 A much more complete survey of the literature is found in 
 \cite{mamotreto}.  We remark that by now numerical calculations
 of quasiperiodic circles can be combined with a-posteriori analysis
(based on Nash-Moser implicit function theory) to obtain mathematically rigorous 
computer assisted proofs \cite{alexCAPKAM}.
 Several additional comments further put the present work into perspective.

 \begin{remark}[Generality] \label{rem:generality} 
 {\em
 Since both the method of weighted averages and the Parmaeterization Method
 generalize to higher dimensional tori for (symplectic) maps in higher dimensions
 -- and even to invariant tori for Hamiltonian systems -- our whole approach 
 generalizes as well.  Nevertheless, we focus on the case of invariant circles to minimize 
technical complications (multivariable Fourier series, rotation vectors, 
Parameterization Method for vector fields, et cetera).  
 }
 \end{remark}

 \begin{remark}[The introduction of a global unfolding parameter] \label{rem:unfolding}
 {\em
  Since any rotation of an invariant circle
  is again invariant, the conjugacy equation defining a parameterization has always a 
  one dimensional family of solutions.  Because of this, the parameterization method 
  for invariant circles is generally degenerate (i.e. there is not a 
  unique parameterization).  
  Of course this is the same non-uniqueness 
  found in the functional analytic set up for periodic orbits for vector fields, and the same 
  solution works: namely, we impose a Poincare type phase condition.  Appending a scalar
  constraint however results in more equations 
  than unknowns.  If the system were dissipative, so that invariant circles are isolated
   in phase space, we would treat the rotation number as a new unknown to rebalance the 
   system.  This does not work for the area preserving maps studied in the present work, 
   as solutions are expected to occur in Cantor sets, and are hence not isolated in phase space.  
   
    In previous works this problems is solved by 
   ``unfolding'' the linearized equations during the Newton iteration.  This requires 
   an infinite sequence of unfolding parameters, one at each step, and a separate
   argument is required to show that the unfolding parameters accumulate to zero.
   In the present work we we introduce a more global unfolding parameter for the 
   parameterization method, which balances the system on the level of the 
   full nonlinear functional equation. The idea is geometric and utilizes the 
   area preservation in a simple way.
   }
  \end{remark}

  \begin{remark}[Use of composition free parameterization of periodic systems of invariant circles] \label{rem:multipleShooting}
  {\em
 We generalize the parameterization method for invariant circles so that it applies to 
 invariant sets consisting of $k$ disjoint circles.  Each orbit in such a set
 visits each of the circles in some order, and each orbit
 is dense in the collection of circles.   We 
 develop a functional analytic multiple shooting scheme which leads to a system 
 of coupled equations in Fourier space describing the collection of circles.
 Our approach is inspired by the multiple shooting parameterization method
 developed in \cite{MR3706909} for studying stable/unstable manifolds attached 
 to periodic orbits of maps.  
 The main advantage these approaches is that they ``unwarp'' function compositions,
 and the nonlinearity of the resulting functional equations is no more complicated 
 than that of the original map.
 }
 \end{remark}
 
% \begin{remark}[Numerical continuation with respect to rotation number and Sobolev norms] \label{rem:numericalContinuation}
% {\em
%
% }
% \end{remark}
 
The remainder of the paper is organized as follows.  
In Section \ref{sec:background} we review some basic facts about 
invariant circles/rotation numbers, as well as results on weighted Birkhoff
averages and the parameterization method.  
In Section \ref{sec:numericalRecipe} we outline our 
numerical recipe, and Section \ref{sec:examples}
deals with numerical examples.  Section \ref{sec:continuation}
shows how these ideas can be combined with numerical continuation 
to compute families of invariant tori up to the point of breakdown.
Section \ref{sec:conclusions} summarizes the paper.

\section{Invariant circles: weighted averages and the parameterization method} 
\label{sec:background}
In this section we review material pertaining to invariant circles
which weighted averages, and the parameterization method, 
which --while standard-- is not to the best of our knowledge 
collected together in one existing reference.  We suggest 
that reader rapidly skim Section \ref{sec:background} before 
jumping ahead to Section \ref{sec:numericalRecipe} --
referring back to the present section only as needed.

\subsection{Homeomorphisms of the circle and their rotation number} \label{sec:circleMaps}  
Let $T \colon \mathbb{S}^1 \to \mathbb{S}^1$ be a homeomorphism of the circle and 
let $\pi \colon \mathbb{R} \to \mathbb{S}^1$ denote the canonical covering map 
defined by 
\[
\pi(x) = x \, \mbox{mod} 1,
\]
mapping a real number $x$ into $[0, 1)$, by discarding the integer part.
 We interpret $\theta \in [0,1)$ as the angle 
describing a point on the unit circle.  Note that $\pi(x + m) = \pi(x)$ for all $x \in \mathbb{R}$.

For $\rho \in [0, 1)$, define 
$R_\rho \colon \mathbb{S}^1  \to \mathbb{S}^1$ by 
\[
R_\rho(\theta) = \theta + \rho \, (\mbox{mod} 1).
\]
We say that a homeomorphism $T \colon \mathbb{S}^1 \to \mathbb{S}^1$ is 
topologically conjugate to the rotation $R_\rho$ if there exists a homeomorphism 
$h \colon \mathbb{S}^1 \to \mathbb{S}^1$ so that 
\[
T(h(\theta)) = h(R(\theta)), 
\]
for all $\theta \in [0,1) = \mathbb{S}^1$.  If $\rho$ is irrational, 
we say that $T$ is conjugate to irrational rotation.

A continuous map $G \colon \mathbb{R} \to \mathbb{R}$ is a \textit{lift} of $T$
if  
\[
(\pi \circ G)(\theta) = (T \circ \pi)(\theta),
\]
for all $\theta \in [0,1) = \mathcal{S}^1$.  It can be shown that every 
continuous map of the circle has a lift, and that $G$ is a lift
of a continuous circle map if and only if there is a $\bar m \in \mathbb{Z}$
such that 
\[
G(x+1) = G(x) + \bar m, 
\]
for all $x \in \mathbb{R}$.  
It follows that 
\[
G(x + m) = G(x) + \bar m \cdot m,
\]
for all $x \in \mathbb{R}$ and every $m \in \mathbb{Z}$.

The rotation number of the homeomorphism 
$T$ is defined by 
\[
\rho = \rho(T) = \lim_{n \to \infty} \frac{G^n(x) - x}{n}, 
\]
where $G$ is a lift of $T$.
It is a classical result (due to Poinca\'{r}e) that the rotation number exists, 
and is independent of both the base point $x$ and the lift $G$.   
 Indeed, it can be shown that $\rho$ is invariant under
continuous change of coordinates (homeomorphism).  
That is, the rotation number is a topological invariant of the
map $T$.

The rotation number has dynamical significance.  For example,
if $\rho(T)$ is a rational number, so that $\rho = p/q$ for some
$p \in \mathbb{Z}$ and $q \in \mathbb{N}$, then $T$
has an orbit of period $q$. 
We focus on the case were $\rho$ is irrational, in which case 
 the Denjoy theorem states the following: if 
$T$ is at least $C^2$, then then $T$ is topologically conjugate to 
the rotation map $R_\rho$.  In this case, it is clear that 
every orbit of $T$ is dense in the circle. 
More detailed discussion of circle maps is found in 
Chapter 2 of \cite{MR1792240} or Chapter 1.2 of 
\cite{MR1326374}.

Note that the rotation number can be computed by averaging 
angles as follows.  
Choose $\theta_0 \in [0, 1) = \mathbb{S}^1$ and define the length 
$N$ orbit segment for $\theta_0$ under $T$ by 
\[
\theta_k = T^k(\theta_0), \quad \quad \quad \mbox{for}  \quad k = 0, \ldots, N.
\]
Using the properties of the covering map, and adding and subtracting 
along the orbit of $\theta_0$, we have that 
\begin{equation} \label{eq:rotNumAvg}
\rho = \lim_{N \to \infty} \frac{1}{N} \sum_{n=0}^{N-1} \left(\theta_{n+1} - \theta_n \right),
\end{equation}
where the \textit{positive difference} of 
two points $\theta, \sigma \in [0, 1) = \mathbb{S}^1$, is 
defined to be
\begin{equation} \label{eq:circleDiff}
\theta - \sigma = \mbox{min} \left(
|\theta - \sigma|, |1 + \sigma - \theta| 
\right).
%\theta - \sigma = \begin{cases}
%\theta - \sigma & \mbox{if } \theta - \sigma < 0.5 \\
%1 + \theta - \sigma & \mbox{if } \theta - \sigma > 0.5 \mbox{ and } \theta < \sigma \\
%-(1 + \sigma - \theta) & \mbox{if } \theta - \sigma > 0.5 \mbox{ and } \sigma < \theta
%\end{cases}
\end{equation}

\subsection{Weighted Birkhoff averages and the rotation number}
\label{sec:rotationNumbers}
The rotation number of a circle map
can be written as an average via Equation \eqref{eq:rotNumAvg}, 
and Ergodic theory is the branch of dynamical systems theory
dealing with averages.  We review some basic 
convergence results from Ergodic theory.  

Let $(X, \Sigma, \mu)$ be a measure space  with $\mu(X) = 1$.
The self map $T \colon X \to X$ is a measure preserving 
transformation of $X$ if $T$ is a measurable function with 
$\mu(T^{-1}(A)) = \mu(A)$ for all $A \in \Sigma$.
The transformation $T$ is \textit{ergodic} if for every $A \in \Sigma$ having
$T^{-1}(A) = A$, it is the case that either $\mu(A) = 0$ or $\mu(A) = 1$.  
%That is, $T$ is ergodic means that $\emptyset$ and $X$ are the only 
%measurable invariant subsets. 
Ergodicity is invariant under homeomorphism, 
in the sense that if $T \colon \mathbb{S}^1 \to \mathbb{S}^1$ is ergodic
and $h  \colon \mathbb{S}^1 \to \mathbb{S}^1$ is a homeomorphism, then 
$T \circ h$ ergodic.

As an example, it is straightforward to show that if $\rho \in [0, 1)$ is irrational, then 
the circle rotation $R_\rho$ is ergodic with respect to Lebesgue 
measure on the circle.  It follows that any circle map topologically conjugate
an irrational rotation is ergodic.

An \textit{observable} on a $X$ is measurable, real (or complex) valued function 
on $X$.  Let $L^1(X, \mu)$ denote that set of all $\mu$-integrable functions
from $X$ to $\mathbb{R}$ (or $\mathbb{C}$).  That is,
the set of all integrable observables.  
For any $f \in L^1(X, \mu)$,   
the Birkhoff ergodic theorem states that 
if $T \colon X \to X$ is ergodic, then 
\begin{equation} \label{eq:ergodicTheorem}
\lim_{n \to \infty} \frac{1}{n} \sum_{k = 0}^{n-1} f(T^k(x)) = \int_X f \, d \mu, 
\end{equation}
for $\mu$-almost ever $x \in X$ \cite{MR1326374}.
That is, the time average of 
the observable $f$ along the $T$-orbit of almost any point $x$, is equal to the 
spatial average of the function $f$ over $X$.   The sum on the left is referred to 
as the Birkhoff average of $f$.  

We are interested in the case when
 $X = \mathbb{S}^1$ and $\mu$ is Lebesgue measure 
on the circle.  Consider an orientation preserving homeomorphism 
$T \colon \mathbb{S}^1 \to \mathbb{S}^1$
(which is measurable by virtue of being a continuous map), and 
suppose that $\rho(T)$ is irrational.  
Define the observable
$\tau \colon \mathbb{S}^1 \to \mathbb{R}$ to be the map
that includes $\theta \in [0, 1) = \mathbb{S}^1$ into the real numbers, 
and the observable $f \colon \mathbb{S}^1 \to \mathbb{R}$ by 
\[
f(\theta) = \tau(T(\theta) -\theta). 
\]
Noting that $f \in L^1(\mathbb{S}^1, \mu)$ we have, by the Birkhoff 
ergodic theorem, that  
\begin{equation} \label{eq:ergodicRotationNum}
\rho(f) = \lim_{n \to \infty} \frac{1}{n} \sum_{k = 0}^{n-1} f(T^k(\theta_0))
= \int_{\mathbb{S}^1} f  \, d \mu, \quad \quad \quad 
\mbox{for almost all } \theta_0 \in \mathbb{S}.
\end{equation}

The utility of the formula given in Equation \eqref{eq:ergodicRotationNum}
is limited in applications by the fact that the sum
suffers from slow (linear) convergence properties.  That is, there 
exists $C > 0$ so that 
\[
\left| \rho(f) - \frac{1}{N} \sum_{k = 0}^{N-1} f(T^k(\theta)) \right|
\leq \frac{C}{N}.
\]
This can be seen by noting that, when $f$ is ergodic, the average in the middle
of Equation \eqref{eq:ergodicRotationNum} is a uniform discretization of the 
integral on the right. 
Then, for example, if we desire fifteen correct digits in the approximation of 
the rotation number, we require approximately $N = 10^{15}$ iterations
of the map $T$.  In addition to being time prohibitive, such a calculation
is numerically unstable due to round off errors.  

In \cite{MR3755876,MR3718733,MR3526529}, the authors show that
if $\rho$ is Diophantine and $f$ is $C^\infty$,
then a much faster convergence rate obtained by taking appropriate 
weighted sums in the Birkhoff averages.
To state the result, define the weights
\begin{align*}
\hat{w}_{n,N} = \dfrac{w\left(\frac{n}{N}\right)}{\sum_{j=0}^{N}w\left(\frac{j}{N}\right)}
\end{align*}
where $w(t) = \exp\left(-1/(t(1-t))\right)$. 
The weighted Birkhoff average is defined by  
\begin{align*}
WB_N(T, f)(\theta) = \sum_{n=0}^{N-1}\hat{w}_{n,N} f(T^n(\theta)).
\end{align*}
Heuristically, this scheme weights more heavily the ``typical'' terms in the middle
of the sequence, avoiding ``boundary effects'' due 
the fact that we average only a finite orbit segment.
This is related to choosing a "good convolution kernel" in the integral
on the right hand side of the ergodic theorem (Equation \eqref{eq:ergodicRotationNum})
\cite{MR3526529,MR3718733,MR3755876}.

The qualitative comments above are made precise in 
in \cite{MR3718733},
and it is shown that 
$WB_N(T,f)(x)$ converges faster than any polynomial,
provided that $\rho(T)$ is ``irrational enough''.  More precisely, we
say that $\rho \in [0, 1)$ is Diophantine if there exist $C, \tau > 0$ so 
that 
\[
|n \rho - m |  \geq \frac{C}{n^{1 + \tau}}, \quad \quad \quad
\mbox{for all } m, n \in \mathbb{N}, n \neq 0.
\] 
This make precise the notion that $\rho$ is not well approximated 
by any rational number.  The main result of   \cite{MR3718733} is that 
if $T$ and $f$ are $C^\infty$, and 
$\rho$ is Diophantine, then for each $M \in \mathbb{N}$ there is a 
$C_M > 0$ so that 
\begin{equation}
\left| \int_{\mathbb{S}} f \, d \mu  -  
WB_N(T, f)(\theta) \right| \leq \frac{C_M}{N^M}.
\end{equation}
Moreover, the convergence is uniform in $\theta$.
Then, in this case, the average converges faster than any polynomial.

%\subsubsection{Example}
%Consider the rotation matrix 
%\[
%R_\rho = 
%\left(
%\begin{array}{cc}
%\cos(2\pi \rho) & - \sin(2 \pi \rho) \\
%\sin(2 \pi \rho) & \cos(2 \pi \rho)
%\end{array}
%\right), 
%\]
%and take 
%\[
%\rho = \frac{1}{2 \pi} \frac{1+\sqrt{5}}{2} \approx 0.257518107400242,
%\]
%the \textit{golden ratio} of $2 \pi$. Note that every circle is invariant 
%under $R$, and have rotation number $\rho$.  
%

\subsection{Invariant circles for area preserving maps} \label{sec:rotNumMaps}
As an application of the smooth ergodic theory discussed in Section \ref{sec:rotationNumbers}, 
we return to the main problem of the paper: 
computing invariant circles for planar dynamical systems.  
To begin making things precise, 
let $\Omega \subset \mathbb{R}^2$ be an open subset of the plane and suppose that 
$F \colon \Omega \to \Omega$ is a smooth, orientation preserving
 diffeomorphism.   Suppose that $\Gamma \subset \Omega$ is a 
$C^\infty$ simple closed invariant curve for $F$, so that 
\[
F(\Gamma) = \Gamma, 
\]
with equality in the sense of sets.   

Restricting $F$ to $\Gamma$ defines a smooth and orientation 
preserving homeomorphism of the circle, which we denote by 
$T \colon \mathbb{S}^1 \to \mathbb{S}^1$.  Since $F$ and $\Gamma$ 
are smooth, so is $T$.  
Following \cite{MR3718733,MR3526529,MR3755876}
we are interested 
 the case where $T$ is conjugate to an irrational rotation.  
To signify the importance of this case we make the following definition: 
we say that $\Gamma$ is a \textit{quasi-periodic invariant circle} for $F$
if the dynamics generated by $F$ restricted to $\Gamma$ -- that is 
the dynamics of $T$ --
are topologically conjugate to an irrational rotation.  
For a given quasi-periodic invariant circle $\Gamma$, we are interested
in determining the rotation number 
of $T$, from finite data for iterates of $F$.

To this end, 
choose $N \in \mathbb{N}$ and 
suppose then that $p_0 \in \Gamma \subset \Omega$.
Define the orbit sequence of length $N$ recursively by 
\begin{equation} \label{eq:def_pj}
p_j = F(p_{j-1}), 
\end{equation} 
for $j = 1, 2, 3, \ldots, N$.  We write 
\[
\mbox{orbit}_{N,F}(p_0) = \left\{ p_j \right\}_{j = 0}^N,
\]
to denote this set.  We convert this to angular data on the circle
as follows.   Let $q_0 \in \Omega$ denote a point inside the 
curve $\Gamma$, and compute the vectors 
\begin{equation} \label{eq:def_pj}
\left(
\begin{array}{c}
x_j \\
y_j
\end{array}
\right) = 
\xi_j = p_j - q_0.
\end{equation}
Define 
\[
\theta_j = \frac{\mbox{atan4}(y_j, x_j) }{2 \pi},
\]
for $j = 0, 1, \ldots, N$.  Here $\mbox{atan4}$ is the four 
quadrant arctangent function which returns the angle 
between $\xi_j$ and the $x$-axis, with the angle 
taken between $0$ and $2 \pi$.  This gives an 
explicit projection of the 
dynamics into $\mathbb{S}^1$.  
Applying the formula developed in Equation \eqref{eq:ergodicRotationNum}, 
we have that 
\[
\rho_N = \sum_{n = 0}^{N-1} \hat{w}_{n,N} (\theta_{n+1} - \theta_n),
\]
rapidly converges to $\rho$, the rotation number of $T$,   
(Again, subtraction for points on the circle is as defined in Equation 
\eqref{eq:circleDiff}).

\begin{remark}[Rotation number as a chaotic/quasiperiodic indicator] \label{rem:test}
{\em
It is important to note that in application problems, we do not actually
know how to choose a $p_0$ on an invariant circle $\Gamma$.
Rather this is, in practice, the problem we are trying to solve.
How then do we decide when an orbit segment is sampled
 from a quasi-periodic invariant circle? A simple answer (which is surprisingly 
 useful in practice) is to examine plots of orbit segments of length $N$,
 for a number of different values of $N$.  Then, one checks visually
 if the plotted orbits appear to densely fill a simple closed curve.

A more sophisticated approach is considered in 
  \cite{MR4104977,MR4322369}, and we 
sketch the idea here.  Consider a point $p_0 \in \Omega \subset \mathbb{R}^2$,
choose an increasing finite sequence of natural numbers 
$0 < N_1 < N_2 < \ldots < N_K$, define the orbit segment 
$\mbox{orbit}_{N_K, F}(p_0)$, the projected angles $\theta_0, \ldots, \theta_{N_K}$,
and compute the approximate rotation numbers
\[
\rho_{N_j} = \sum_{n = 0}^{N_j} \hat{w}_{n,N}(\theta_{n+1} - \theta_n),
\]
for $j = 1, 2, \ldots, K$.  If the $\rho_{N_j}$ converge numerically, this 
provides strong evidence that $p_0$, and hence
the points in $\mbox{orbit}_{N_K, F}(p_0)$, are sampled from a 
quasi-periodic invariant circle $\Gamma$.  If on the other hand the 
sequence $\rho_{N_1}, \rho_{N_2}, \ldots, \rho_{N_K}$ oscillates
randomly, then the orbit of $p_0$ is more likely sampled from 
a stochastic zone rather than a quasi-periodic orbit.  
}
\end{remark}

\begin{remark}[Elliptic equilibria and KAM phenomena] \label{rem:KAM}
{\em
A common mechanism which gives rise to invariant circles
is the KAM scenario for an elliptic
fixed point.   To formalize the discussion, let
$F \colon \mathbb{R}^2 \to \mathbb{R}^2$ be 
an orientation preserving, $C^2$ diffeomorphism of the plane, and suppose that  
$p \in \mathbb{R}^2$ is an elliptic fixed point of $F$.   That is, we assume
that $F(p) = p$, and that the eigenvalues of $DF(p)$, 
$\lambda_{1,2} = e^{\pm i \rho }$, are on the unit circle.
If $\rho$ is irrational, then the linearized dynamics at $p$ 
consist of concentric invariant circles, on which orbits are 
dense.  The dynamics in a small neighborhood of $p$ can be 
analyzed as nonlinear perturbation of the linear map $DF(p)$. 
The main question of KAM theory in this context is:
which if any of the invariant circles survive the perturbation? 

The answer depends on the number theoretic properties 
-- more precisely the Diophantine properties --   
of $\rho$, and on some nonlinear non-degeneracy, 
or twist conditions on the higher derivatives of $F$ at $p$. 
(Recall that the Diophantine constants measure ``how irrational'' 
a real number is).  Heuristically speaking, the typical situation is 
that a Cantor set of invariant circles survives.  
Moreover, a similar picture, in the neighborhood of an elliptic 
periodic $K$ orbit, gives rise to period $K$ systems of invariant 
circles.  From the point of 
view of the present paper, the main observation is that invariant 
circles with irrational dynamics are natural in area preserving
maps.  An excellent reference is 
\cite{DLL01}.
}
\end{remark}

\subsubsection{Weighted Birkhoff averages and the Fourier coefficients of the embedding}  
\label{sec:weightsFourier}

Suppose that $\Gamma$ is a quasi-periodic invariant circle for the 
diffeormophism $F \colon \Omega \to \Omega$.  
Another application of the smooth ergodic theory discussed in Section 
\ref{sec:rotationNumbers} is to compute the Fourier coefficients of
a lift/parameterization for $\Gamma$.

To be precise, we seek a period one
function $K \colon \mathbb{R} \to \mathbb{R}^2$
so that $\mbox{image}(K) = \Gamma$, with $\Gamma$
quasi-periodic.  Indeed, since the dynamics on $\Gamma$
are conjugate to $R_\rho$ (with $\rho$ the rotation number 
for $\Gamma$) we look for the conjugating map $K$.
That is, we require that 
\begin{equation} \label{eq:conjEq}
F(K(\theta)) = K(\theta + \rho), 
\end{equation}
and to fix the phase of $K$ we impose $K(0) = p_0$.  
The geometric meaning of Equation \eqref{eq:conjEq} 
is illustrated in Figure \ref{fig:conj}.

\begin{figure}[ht!]
\centering
\includegraphics[scale=0.35]{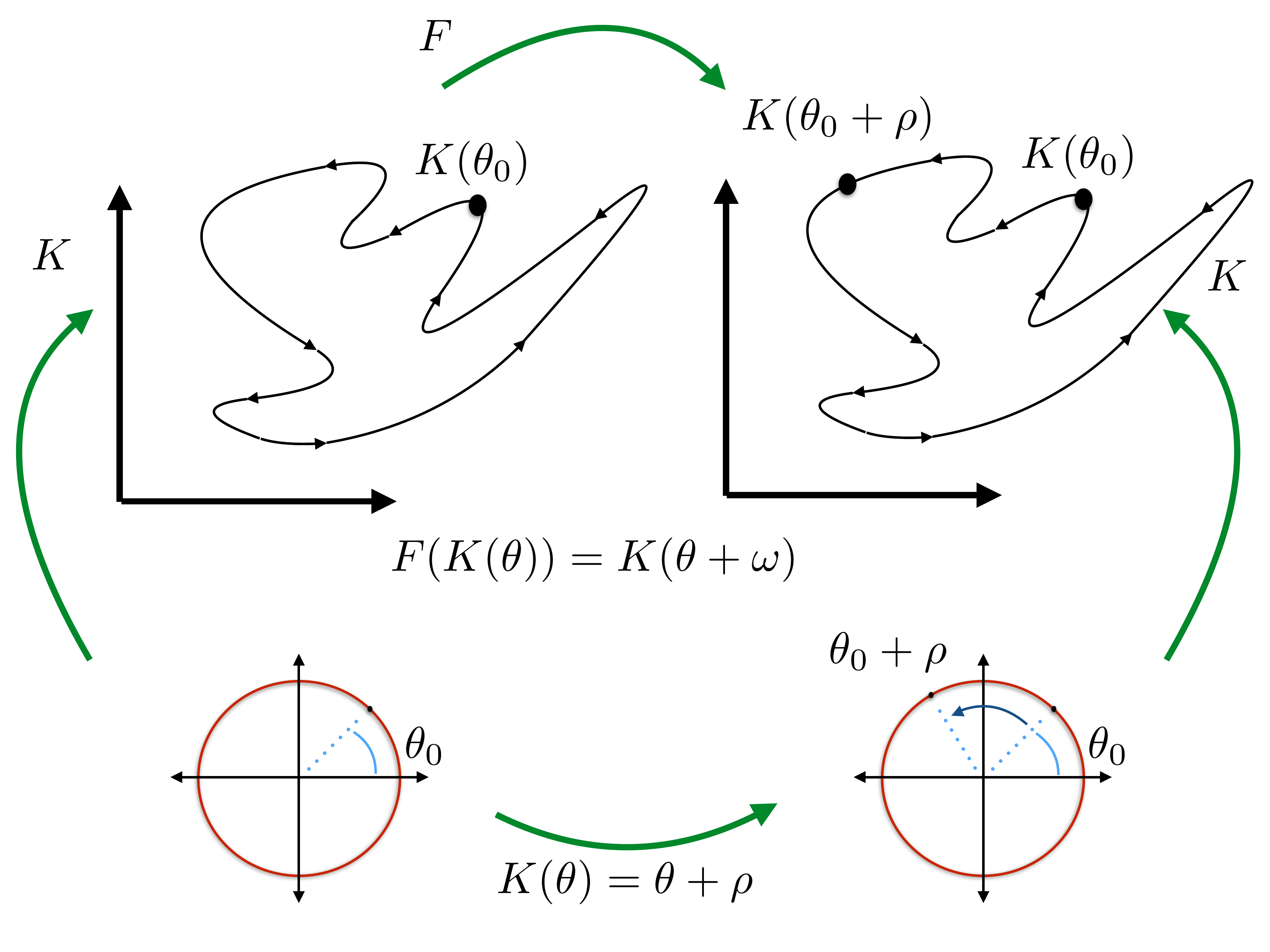} 
\caption{Topological conjugacy to rotation: 
  Here $K \colon \mathbb{S} \to \mathbb{R}^2$ is an 
embedding of the circle.  The image of $K$ is invariant under $F$, 
in the sense that $F \circ K$ is a reparameterization of the curve $K$.
In fact, the reparameterization is rotation by an angle $\rho$.
The invariance equation $F \circ K = K \circ f$ expresses the 
fact that the above diagram commutes, 
meaning that the dynamics on $K$
generated by $F$ are conjugate to the dynamics on the circle generated
by $f(\theta) = \theta + \rho$. } \label{fig:conj}
\end{figure}

Since $\Gamma$ is a smooth curve, the map $K$ 
is smooth and has convergent Fourier series which we 
denote by 
\begin{align*}
K(\theta) = \begin{pmatrix}
K^1(\theta) \\
K^2(\theta)
\end{pmatrix} = 
\begin{pmatrix}
\sum_{n\in\Z} a_n e^{2\pi i n \theta} \\
\sum_{n\in\Z} b_n e^{2\pi i n \theta}
\end{pmatrix}
=\sum_{n\in\Z} 
\begin{pmatrix}
a_n \\
b_n
\end{pmatrix}
 e^{2\pi i n \theta}
=\sum_{n\in\Z} k_n e^{2\pi i n \theta}
\end{align*}
where
\begin{align*}
k_n & = \int_0^1 K(\theta) e^{-2\pi i n \theta} \, d\theta.
\end{align*}

The idea is to treat each Fourier coefficient as a 2-vector of observables
for the underlying circle map $T$.  This can be done, exploiting the 
fact that Fourier coefficients are defined in terms of integrals and 
applying the weighted Birkhoff averages of \cite{MR3526529}.
Inductively applying Equation \eqref{eq:conjEq},
we have that
\[
p_k = F^k(p_0)  = F^k(K(\theta_0)) = K(\theta_k),
\] 
with 
\[
\theta_k = \theta_0 + k \rho,
\]
for $k = 0,1,2, \ldots, N$.  
Then 
\begin{align*}
k_n & = \int_0^1 K(\theta) e^{-2\pi i n \theta} \, d\theta \\
&= \lim_{N \to \infty}\sum_{k=0}^{N-1} \hat{w}_{k,N} K(\theta_k) e^{-2\pi i n \theta_k} \\
&= \lim_{N \to \infty}\sum_{k=0}^{N-1} \hat{w}_{k,N} K(\theta_0 + k \rho) e^{-2\pi i n (\theta_0 + k \rho)} \\
&= \lim_{N \to \infty}\sum_{k=0}^{N-1} \hat{w}_{k,N} F^k(K(\theta_0)) e^{-2\pi i n (\theta_0 + k \rho)} \\
 & =  \lim_{N \to \infty}\sum_{k=0}^{N-1} \hat{w}_{k,N} p_k e^{-2\pi i n (\theta_0 + k \rho)} \\
   & =  \lim_{N \to \infty} e^{-2\pi i n \theta_0}   \sum_{k=0}^{N-1} \hat{w}_{k,N} p_k e^{-2\pi i n k \rho} 
\end{align*}
Then
\begin{align*}
k_n \approx e^{-2\pi i n \theta_0}  \sum_{k = 0}^{N-1} \hat{w}_{k,N} p_k e^{-2\pi i nk \rho}. %Or something like this...
\end{align*}

The major sources of error in this approximation of the Fourier coefficient $k_n$ are threefold.
First, the limit as $N \to \infty$ is approximated by computing a finite, rather than an infinite sum.   
Second, there is the error from the approximated rotation number used to compute the coefficients, 
that is we use $\rho_N$ for some high enough $N$ to approximation $\rho$.  
Third, the trajectory $\{p_n\}_{n=0}^N$ is only near a quasiperiodic orbit, generated 
as it is by numerically iterating the map $F$. 
Of  course, in the end, the parameterization $K$ is approximated using a finite number 
of Fourier modes.

\subsection{The Parameterization Method} \label{sec:parmMethod} 
The parameterization method is a general functional analytic framework for studying 
invariant objects in discrete and continuous time dynamical 
systems.  While the method has roots in the classical works 
of Poincare, Darboux, and Lyapunov,  a complete theory
for fixed points of infinite dimensional nonlinear maps on Banach spaces
emerged in the three papers of Cabr\'{e}, Fontich, and de la Llave 
 \cite{MR1976079,MR1976080,MR2177465}.    
The corresponding theory for invariant tori (quasi-periodic motions)
and their whiskers (stable/unstable fibers) for skew product dynamical systems
is developed in the three papers
by Haro and de la Llave  \cite{MR2289544,MR2240743,MR2299977}.
Since its introduction in the papers just cited, the method has been expanded and 
applied by a number of authors, so that a complete overview of the literature 
is a task beyond the scope of the present work.  
The interested reader will find an informative and lively discussion of the 
history of the method in Appendix B of \cite{MR2177465}.  
Moreover, the recent book on the topic 
by Haro, Canadell, Figueras, Luque, and Mondelo \cite{mamotreto}
contains detailed discussion of the method, a thorough review of the literature, 
and many detailed example applications.

\subsubsection{Parameterization method for an invariant circle in the plane} 
\label{sec:parmMethodInvCirc}

In the case of invariant circles, the 
main idea behind the parameterization method is to 
treat Equation \eqref{eq:conjEq} as an equation for an 
unknown smooth $1$-periodic function $K \colon \mathbb{R} \to \mathbb{R}^2$, 
and to attempt to solve in an appropriate function space via a Newton iteration 
scheme.  Since the Newton method is based in the implicit function theorem, 
it is essential that we look for an isolated solution of Equation \eqref{eq:conjEq}.
Note however that any rotation of a solution is again a solution, and it is 
necessary to fix a phase condition to isolate.  
In the present work we fix the phase by requiring that $K(0)$ lies in 
a fixed (by us at the outset of the discussion) line in the plane.  
That is, we choose vectors $\bar p, \eta \in \mathbb{R}^2$ and 
add the constraint equation
\[
\left< \bar p - K(0), \eta \right> = 0,
\]
where $< \cdot, \cdot >$ is the usual inner product in $\mathbb{R}^2$.
The idea here is that $\bar p$ and $\eta$ determine a line $\ell$ 
transverse to $\Gamma$ and we require $K$ to map $\theta = 0$
into the line $\ell$, thus locking down the phase of the parameterization.

The issue now comes when we consider the resulting system of 
equations 
\begin{align*}
F(K(\theta)) =  K(\theta + \rho). \\
\left< \bar p - K(0), \eta \right> = 0 
\end{align*}
which is clearly two equations in one unknown $K$.  
To balance the system we introduce a saclar unfolding parameter
$\beta$.  That is, we consider the system of equations 
\begin{align*}
F(K(\theta)) = (1 + \beta) K(\theta + \rho) \\
\left< \bar p - K(0), \eta \right> = 0,
\end{align*}
as two equations in two unknowns $K$ and $\beta$.
This idea is inspired by similar techniques for balancing
the systems of equations describing periodic orbits 
in Hamiltonian systems.  See for example
\cite{MR2003792,MR1870260}.
As with any work involving unfolding parameters, we have to address
the relationship between the original unbalanced system of equations 
and the unfolded system. This is the content of Lemma \ref{lem:unfolding},
which shows that solutions of the unfolded system satisfy
the original equations.

Let $C^k_p(\mathbb{R})$ denote the space of smooth, period-$1$ 
functions, with $k > 1$. (In our applications $k = \omega$)
 and define the nonlinear mapping 
$\Psi \colon \mathbb{R} \times C^k_p(\mathbb{R}) \to \mathbb{R} \times C^k_p(\mathbb{R}) $
by 
\begin{equation}\label{eq:defPsi}
\Psi(\beta, K) = \left(
\begin{array}{c}
\left< \bar p - K(0), \eta \right>  \\
F(K(\theta)) - (1+\beta) K(\theta + \rho)
\end{array}
\right).
\end{equation}
Let $\mathbf{0}$ denote the zero function.  
We have the following.

\begin{figure}[ht]
\centering
\includegraphics[scale=0.15]{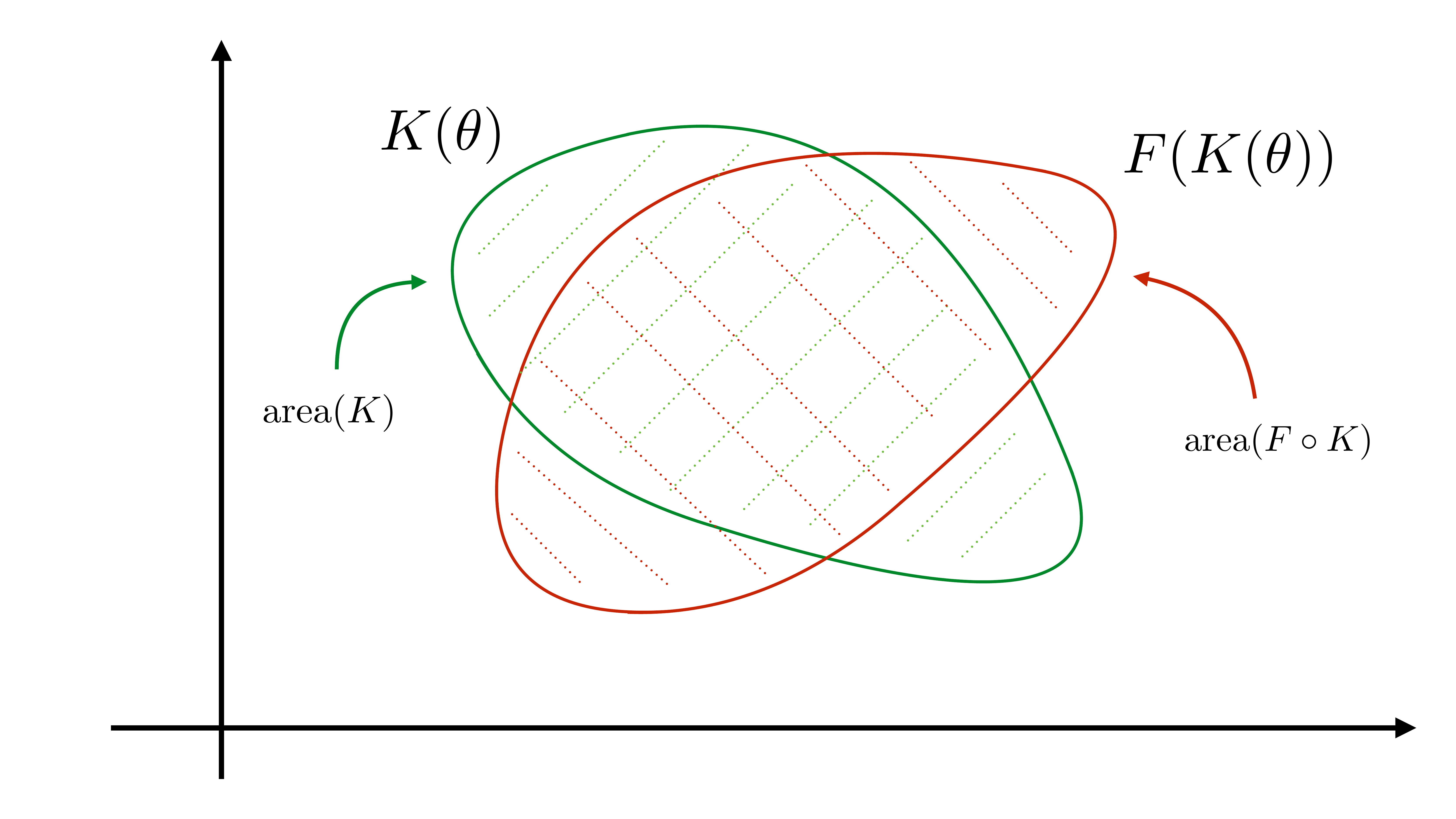} 
\caption{asdaf} \label{fig:areaConserved}
\end{figure}

\begin{lemma}[$\beta$ unfolds Equation \eqref{eq:defPsi}] \label{lem:unfolding}
Suppose that $K_* \in C_p^k(\mathbb{R})$ and $\beta_* \in \mathbb{R}$ have 
\[
\Psi(K_*, \beta_*) = 
\left(
\begin{array}{c}
\mathbf{0} \\
0
\end{array}
\right).
\]
Then $\beta_* = 0$, and
$K$ conjugates the dynamics on $\mbox{image}(K)$ generated by 
$F$ to the rotation map $R_{\rho}$.
\end{lemma}
\begin{proof}
Suppose that $K_*$ and $\beta_*$ provide a zero of $\Psi$.  Then 
\begin{equation} \label{eq:lemEq1}
F(K_*(\theta))  = (1+\beta_*) K_*(\theta + \rho).
\end{equation}
Let $\Gamma$ denote the curve parameterized by $K_*$,
and $\tilde \Gamma = F \circ \Gamma$ be the curve parameterized 
by $F \circ K_*$.  
Note that $\tilde \Gamma$ is diffeomorphic to $\Gamma$, due to the 
assumption that $F$ is a diffeomorphism, and that 
$K_*(\theta + \rho)$ is just a reparameterization of the curve $\Gamma$, 
with different phase.  

For $K \in C_p^k(\mathbb{R})$, consider the integrals 
\begin{align*}
A_1(K) & =  \frac{1}{2} \int_\Gamma
K_1 \, dy - K_2 \, dx \\
 & =  \frac{1}{2} \int_0^1 \left(
K_1(\theta) \frac{d}{d \theta} K_2(\theta) - K_2(\theta) \frac{d}{d\theta} K_1(\theta)
\right) \, d \theta
\end{align*}
\begin{align*}
A_2(K) & =  \frac{1}{2} \int_\Gamma
K_1 \circ R_\rho  \, dy - K_2 \circ R_\rho \, dx \\
 & =  \frac{1}{2} \int_0^1 \left(
K_1(\theta+\rho) \frac{d}{d \theta} K_2(\theta+\rho) - K_2(\theta+\rho) \frac{d}{d\theta} K_1(\theta+\rho)
\right) \, d \theta
\end{align*}
and 
\begin{align*}
A_3(K) & =  \frac{1}{2} \int_{\tilde \Gamma}
(F \circ K)_1   \, dy - (F \circ K)_2 \circ R_\rho \, dx \\
 & =  \frac{1}{2} \int_0^1 \left(
F(K(\theta)_1 \frac{d}{d \theta} F(K(\theta)_2 - F(K(\theta)_2 \frac{d}{d\theta} F(K(\theta)_1
\right) \, d \theta.
\end{align*}
To motivate the consideration of these integrals, we note
that if $\Gamma$ is a simple closed curve, then $\tilde \Gamma$
is a simple closed curve as well (as $F$ is a diffeomorphism) and
$A_1, A_2$ would correspond (by Green's theorem) to the area
 enclosed by $\Gamma$.  Similarly,  $A_3$ would be 
 the area enclosed by $\tilde \Gamma$.  See Figure \ref{fig:areaConserved}.
We remark that $A_1, A_2, A_3$ are well defined in general 
as long as $\Gamma$ is closed and , i.e. for all $K$ in $C_p^k(\mathbb{R})$ 
with $k > 1$, by Greens theorem, and that if the curves have self intersections then the 
integrals compute enclosed area with overlap.

Moreover, since $A_1$ and $A_2$ are computed over the same curve
(with different parameterizations) we have that 
\[
A_1 = A_2.
\]
Since $\hat \Gamma$ is diffeomorphic to $\Gamma$, and $F$ is 
an area preserving map in the plane (and hence a symplectomorphism)
we also have that $A_1= A_3$.  

However, integrating both sides of Equation \eqref{eq:lemEq1}, 
gives
\[
A_3(K_*) = (1+\beta_*) A_2(K_*).
\]
Combining this with the fact that
 $A_2(K_*) = A_1(K_*) = A_3(K_*)$, it follows that $\beta_* = 0$.
Heuristically speaking, the area enclosed by $\hat \Gamma$ cannot
be either more or less then the area enclosed by $\Gamma$
(where overlaps are counted correctly in both cases).
From this we obtain that 
\[
F(K_*(\theta)) = K_*(\theta + \rho),
\]
and hence $K_*$ conjugates the dynamics on $\Gamma$ to the 
irrational rotation $R_\rho$.  

%If $\rho$ is irrational, then the orbit of any point $\theta \in [0, 1)$
%is dense on $\mathbb{S}^1$ under $R_\rho$, and 
%it follows that $K_*$ is one-to-one on $[0, 1)$, hence
%$\gamma$ is a simple closed curve.    
%
%
%
%To see this suppose that $\theta_1, \theta_2 \in [0,1)$ have that 
%\[
%K_*(\theta_1) = K_*(\theta_2).
%\]
%Denote this point by $p_*$, and note that, since the dynamics
%on $\gamma$ are conjugate to an irrational rotation, 
%the orbit of $p_*$ is dense, and by the conjugacy we have 
%\[
%K_*(\theta_1 + n \rho) = K_*(\theta_2 + n \rho), 
%\]
%for all $n \in \mathbb{Z}$ (they agree along the full orbit).
%
%
%\textbf{...Proof still needs completion...}

%with $\theta_1 \neq \theta_2 \in \mathbb{S}^1$.  Then $\gamma$ has 
%a self intersection at $K_*(\theta_1)$.  From this we see that 
%\[
%K_*(\theta_1 + \rho) = F(K_*(\theta_1)) = F(K_*(\theta_2)) = K_*(\theta_2 + \rho),
%\]
%and there is another self intersection at $K_*(\theta_1+\rho)$.  Proceeding 
%inductively, we arrive at the conclusion that $\gamma$ has self intersections
%on a dense set of points.  
\end{proof}

\subsection{Newton scheme in Fourier coefficient space} \label{sec:Newton}

Fortified by Lemma \ref{lem:unfolding}, we now seek to solve the Equation 
$\Psi(K, \beta) = 0$, as defined in Equation \eqref{eq:defPsi} for the 
unknown parameterization $K$.
Indeed suppose that $K_0$ is an approximate zero of the 
equation and and choose $\beta_0 = 0$.  
The Newton sequence is given by 
\[
\left(
\begin{array}{c}
\beta_{n+1} \\
K_{n+1}
\end{array}
\right) = 
\left(
\begin{array}{c}
\beta_{n} \\
K_{n}
\end{array}
\right) + 
\left(
\begin{array}{c}
\delta_{n} \\
\Delta_{n}
\end{array}
\right), \quad \quad n \geq 0,
\]
where $(\delta_n, \Delta_n)^T$ is a solution of the linear 
equation 
\begin{equation} \label{eq:homEq}
D \Psi(\beta_n, K_n) 
\left(
\begin{array}{c}
\delta_{n} \\
\Delta_{n}
\end{array}
\right) 
= - \Psi(\beta_n, K_n).
\end{equation}
Here, for $\beta, \delta \in \mathbb{R}$ and $K, \Delta \in C_p^k(\mathbb{R})$ the
Frechet derivative of $\Psi$ has action 
\[
D \Psi(\beta, K) 
\left(
\begin{array}{c}
\delta \\
\Delta
\end{array}
\right) 
= 
\left(
\begin{array}{c}
-\left<\Delta(0), \eta\right> \\
-\delta K(\theta+\rho) + DF(K(\theta)) \Delta(\theta) - (1+\beta) \Delta(\theta + \rho) 
\end{array}
\right).
\]

\begin{remark}[Fast algorithms exploiting the symplectic structure] \label{rem:fastAlg}
{\em
The efficiency of the Newton scheme is improved dramatically via  
the area preserving/symplectic structure of the problem, which 
facilitates reduction of the linear equation \eqref{eq:homEq} to 
constant coefficient, plus a quadratically small error.
This idea is known in the literature as \textit{approximate reducibility}.     
Neglecting the quadratic error, the resulting constant coefficient linear equations 
are easily diagonalized (in Fourier coefficient space).
The reader interested in state of the art algorithms is referred 
to\cite{MR4203515,MR3388608,MR4263036,MR4344967,MR4433114}
We again refer to \cite{mamotreto} for comprehensive discussion.
}
\end{remark}

Since we seek periodic $K$ it is natural to write
make the Fourier \textit{ansatz} 
\[
K(\theta) = \sum_{n \in \mathbb{Z}} \left(
\begin{array}{c}
a_n \\
b_n
\end{array}
\right)
e^{2 \pi i n \theta},
\]
as considered already in Section \ref{sec:weightsFourier}.
Note that translation by $\rho$ is a diagonal operation
in Fourier space, as
\begin{align*}
K(\theta + \rho)& = \sum_{n \in \mathbb{Z}} \left(
\begin{array}{c}
a_n \\
b_n
\end{array}
\right)
e^{2 \pi i n (\theta+\rho)} \\
& = \sum_{n \in \mathbb{Z}}
e^{2 \pi i n \rho}
 \left(
\begin{array}{c}
a_n \\
b_n
\end{array}
\right)
e^{2 \pi i n \theta},
\end{align*}
and that the phase condition can be written as 
\[
\left< \bar p - K(0), \eta \right> = 
\bar p_1 \eta_1 + \bar p_2 \eta _2 -  \sum_{n \in \mathbb{Z}} \left(\eta_1 a_n + 
\eta_2 b_n \right).
\]

The nonlinearity is more complicated, but note that if $K \in C_p^k$ then 
$F \circ K \in C_p^k$ as well, assuming that $F$ is as smooth as $K$.
(For the examples in this paper $F$ is real analytic).  Then $F \circ K$ has 
Fourier expansion 
\[
F(K(\theta)) = \sum_{n \in \mathbb{Z}} (F \circ K)_n e^{2 \pi i n \theta}, 
\]
where the Fourier coefficients $(F \circ K)_n$ depend in a nonlinear way
way on the Fourier coefficients $a_n, b_n$.  
In practice if $F$ is a polynomial map then this dependence is worked out 
by discrete convolutions, as seen in the examples.
Otherwise, the map is computed numerically 
using the FFT.  Indeed, using the FFT, evaluation of the nonlinearity 
is a diagonal operation in grid space.

\subsection{Multiple shooting for period-$d$ systems of invariant circles} \label{sec:multShoot}

We now consider a ``multiple-shooting'' parameterization 
method for studying $d$-periodic systems of quasi-periodic 
invariant sets.  Such a set is the union of $d$ disjoint simple 
closed curves, with the property each point on one curve 
maps to another curve in the system.  The dynamics are 
required to be quasi-periodic.
More precisely, suppose that 
$\Gamma_1, \ldots, \Gamma_d \subset \mathbb{R}^2$
are smooth simple closed curves with 
\begin{eqnarray}
F(\Gamma_1) & = \Gamma_2 \\
F(\Gamma_2) & = \Gamma_3 \\
& \vdots \\
F(\Gamma_{d-1}) & = \Gamma_d \\
F(\Gamma_d) & = \Gamma_1.
\end{eqnarray}
Suppose moreover that, for each $1 \leq j \leq d$, the 
curve $\Gamma_j$ is quasi-periodic for the composition 
map $F^d$.  That is, suppose that for each $1 \leq j \leq d$
 the mapping $F^d$ 
restricted to $\Gamma_j$ is an orientation preserving 
circle homeomorphism with irrational rotation number $\rho_j$.
The situation is illustrated in Figure \ref{fig:perK_invEq}.

Note that compositions of $F$
provide conjugacies between each of the $F^d$ invariant circles 
$\Gamma_j$. 
For example, the map $F$ provides a conjugacy between 
the dynamics on $\Gamma_j$ and $\Gamma_{j+1}$ while 
$F^2$ conjugates $\Gamma_{j}$ to $\Gamma_{j+2}$ and so on.  
Then, since $F$ is a diffeomorphism (and hence a homeomorphism),
 the topological invariance of the rotation number gives that
 $\rho_1 = \ldots = \rho_j$, 
and it is permissible to simply write $\rho$ for the common rotation 
number.

One computational approach for studying this invariant set
would be to apply the parameterization discussed in Section
\ref{sec:parmMethodInvCirc} to the map $F^N$, once for each of the curves 
$\Gamma_j$, $1 \leq j \leq d$.  
This approach however has two major disadvantages:  first,
the computational complexity 
of the composition map evaluation grows exponentially with the 
number of compositions.  
For example if $F$ is polynomial of degree $m$ then $F^2$ is polynomial 
of degree $m^2$ and $F^d$ is polynomial of degree $m^d$.  
The second disadvantage is that one has to compute 
parameterizations of $\Gamma_1, \ldots, \Gamma_d$ separately.  

Instead, we propose a ``multiple shooting'' parameterization 
method for computing the entire period period $d$ systems of 
quasiperiodic curves all at once.
Earlier successful multiple shooting approaches
are developed in the  
work of \cite{MR3706909} for parameterizing  
stable/unstable manifolds attached to period-$d$ orbits of maps,
and in the \cite{MR4446093} for studying invariant objects for
discrete dynamical systems defined by an implicit rule.  
In the current context we look for smooth parameterizations 
$K_1, \ldots, K_d \colon \mathbb{R} \to \mathbb{R}^2$ --
all of period one  -- so that for each
$\theta \in \mathbb{R}$  we have that
\begin{align*}
F(K_1(\theta)) & = K_2(\theta) \\
F(K_2(\theta)) & = K_3(\theta) \\
& \vdots \\
F(K_{d-1}(\theta)) & = K_d(\theta) \\
F(K_d(\theta)) & = K_1(\theta + \rho),
\end{align*}
where $\rho$ is the rotation number associated with any of 
the invariant circles of the composition map 
$F^d$.  

Once again, it is necessary to append a scalar 
constraint to fix the phase of one of the circles -- this 
in turn fixes the phase of each parameterization.  
Appending the phase constraint unbalances the 
system so that it is necessary to introduce an unfolding 
parameter.  Taking these considerations into account, 
we define the operator 
$\Psi_d \colon \mathbb{R} \times C_d^k(\mathbb{R}, \mathbb{R^2})^d \to 
 \mathbb{R} \times C_d^k(\mathbb{R}, \mathbb{R^2})^d$, defined by 
 \begin{equation} \label{eq:sysInvEq}
\Psi_d(\beta, K_1, K_2, K_3, \ldots, K_{d-1}, K_d) = 
 \left(
 \begin{array}{c}
 \left<K_1(0) - \bar p, \eta \right> \\
 F(K_1(\theta)) - K_2(\theta) \\
F(K_2(\theta)) - K_3(\theta) \\
 \vdots \\
F(K_{d-1}(\theta)) - K_d(\theta) \\
F(K_d(\theta)) - (1+\beta) K_1(\theta + \rho),
 \end{array}
 \right)
 \end{equation}
Again, the important thing to stress if that the 
definition of $\Psi_d$ does in to involve any compositions 
of the map $F$.  A Newton method is defined as in 
Section \ref{sec:Newton}.

%
%For fixed $\rho \notin \mathbb{Q}$ suppose that  
%$K_1, \ldots, K_d$ solve this system of equations. It's easy to check that   
%the union of their $d$ images is invariant under $F$, and
%for any $\theta \in \mathbb{R}$, the orbit of $K_1(\theta)$ is dense in 
%the union of the $d$ circles.  The orbit of $K_1(\theta)$ visits
%each of the circles parameterized by $K_1, \ldots, K_d$ in order.
%The situation is illustrated in Figure \ref{fig:perK_invEq}.
%This ``multiple-shooting'' parameterization method does not require 
%computing the composition of $K_j$ with any map other than $F$ 
%(that it it avoids computing $F^N(K_j))$, and the 
%method solves simultaneously for each of the parameterizations
%$K_1, \ldots, K_d$.  

\begin{figure}[t!]
\centering
\includegraphics[scale=0.18]{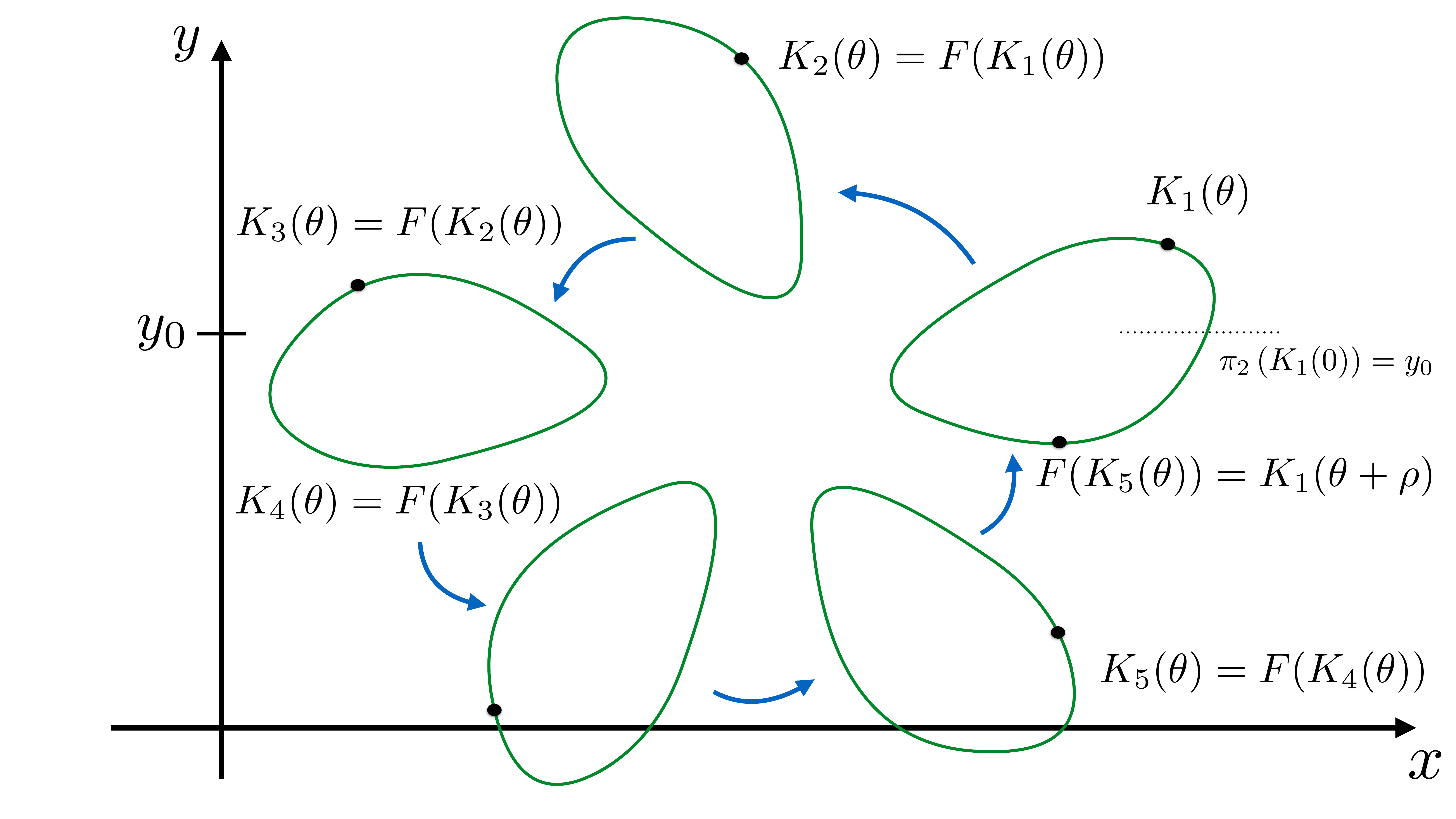} 
\caption{Schematic illustration of a period $5$ invariant circle and the 
resulting parameterization method.} \label{fig:perK_invEq}
\end{figure}

\section{Numerical recipe: initializing the parameterization method via weighted averaging}
\label{sec:numericalRecipe}
Suppose that $\Omega \subset \mathbb{R}^2$ is an open set and let 
$F \colon \Omega \to \Omega$ be a smooth, area preserving map.  
The following algorithm (i) allows us to determine that we have an 
initial condition whose orbit is very likely one or near a quasiperiodic invariant 
circle, (ii) allows us to compute the rotation number efficiently and 
accurately from just the orbit segment data, 
(iii) allows us to easily determine the truncation dimension for the 
finite dimensional Fourier projection of the parameterization,
(iv) leads in a completely 
natural way to an initial guess for the parameterization method which 
can be made as accurate as we like -- hence will definitely converge.
We also (v) have an a-posteriori indicator which allows us to decide 
when Newton has converged.  The following steps constitute the 
main steps of our algorithm.  
\begin{itemize}
\item \textbf{Step 0:} choose $p_0 =(x_0,y_0) \in \Omega$ and $M \in \mathbb{N}$.  
Compute the orbit segment $\mathcal{O}_M = \{p_j\}_{j=0}^M$ defined by 
\[
p_{j+1} = F(p_j), \quad \quad \quad j \in 1, \ldots, M-1.
\]
Now, decide if $\mathcal{O}_M$ is sampled from an invariant circle, or from a stochastic
zone.  This can be done either by graphical inspection, or using the techniques of 
\cite{MR4104977,MR4322369} already mentioned in Remark \ref{rem:test}.
If $\mathcal{O}_M$ appears to be sampled from a quasiperiodic invariant circle, 
then we continue to the next step.  Otherwise, choose a different $p_0$.  
\item \textbf{Step 1: } Compute the rotation number $\rho$ using the 
weighted averaging technique discussed in Section \ref{sec:rotNumMaps}.
Here it is important to obtain as many correct digits as possible.  This can 
be done by increasing $M$ by ten or twenty percent and repeating the 
calculation until numerical convergence in the last digit is observed.  
\item \textbf{Step 2: } Decide how many modes are needed to accurately 
represent $K$.  To do this, we compute the Fourier coefficients 
$(a_n, b_n)$ using the averaging scheme described in Section \ref{sec:weightsFourier}.
However we compute using a much shorter sample (that is we use $M$ much 
smaller than in the rotation number calculation) and sample the modes 
by computing them only for $n = 10 k$, and $k = 1,2, 3, \ldots$.  Using 
this scheme we can rapidly find an $N \in \mathbb{N}$ so that 
$\|(a_n, b_n)\| < \epsilon_{{\tiny \mbox{machine}}}$ for $|n| > N$.
\item \textbf{Step 3: } We now calculate a good initial condition for the 
Newton scheme.  For this we take $N_0$ roughly ten or twenty percent of 
$N$ and compute $(a_n, b_n)$, using moderate accuracy (i.e. $M$ larger than in 
step 2 but smaller than in step 1), for $|n| \leq N_0$  
using the weighted averages of Section \ref{sec:weightsFourier}.
Let's call the resulting degree $N_0$ Fourier polynomial $\tilde K$.
Compute the numerical defect 
\[
\tilde \epsilon = \sup_{\theta \in [0,1]} \left| F(\tilde K(\theta)) - \tilde K(\theta + \rho) \right|.
\]
If $\tilde \epsilon$ is smaller than some tolerance -- which should be less than one but
is usually taken to be between $0.001$ and $0.1$, depending on 
the judgment of the user --
then the initial guess is ``good'' and we set $K_0 = \tilde K$.
If the initial defect is not good enough, then we can increase $N_0$ and try again.  
\item \textbf{Step 4: } Perform the Newton iteration (in the space of $N$-Fourier 
coefficients) as described in Section \ref{sec:Newton}. Iterate the Newton 
scheme until the 
defect
\[
\epsilon_m =  \sup_{\theta \in [0,1]} \left| F(K_m(\theta)) - K_m(\theta + \rho) \right|,
\]
either saturates or decreases below some prescribed tolerance
(usually taken to be some small multiple of machine epsilon).
\end{itemize}

Several remarks are in order.  First, we note that the $C^0$ norm proposed 
for measuring the defect in Steps $3$ and $4$ can be replaced with 
more efficient weighted $\ell_1$ norms, and this involves computations only 
in coefficient space rather than function evaluations.  
We also remark that if the Newton scheme does not converge in 
Step $4$, then we conclude that the initial defect was 
not good enough and go back to step $3$ to refine $N_0$.  

It should also be noted that the defect calculations proposed above 
provide only a heuristic indication of convergence.  
More reliable error bounds for the parameterization method, 
based on a-posteriori Kantorovich-type results, are obtained in 
\cite{MR2289544}.  See also 
\cite{mamotreto}.  Indeed, this kind of a-posteriori analysis can be combined
with deliberate control of round of errors to obtain mathematically 
rigorous computer assisted existence proofs.
Early examples of this kind of argument are found in the work of 
\cite{MR1101369,MR1008096}.
For a more modern treatment, including a thorough discussion of the 
current state of the literature,
we refer the interested reader to the work of \cite{MR3709329}.

\section{Examples}
\label{sec:examples}

\subsection{A quadratic family of maps: area preserving Henon}

As a first example, consider 
the area-preserving H\'enon map, $F: \R^2 \to \R^2$
as described in \cite{Hen69},
and given by the formula
%\begin{align*}
\[
F(x,y) = \left(
\begin{array}{c}
 x \cos(\alpha) - (y - x^2) \sin(\alpha) \\
x \sin(\alpha) + (y - x^2) \cos(\alpha)
\end{array}
\right).
\]
One checks that the determinant of the Jacobian matrix 
$DF(x,y)$ is one for all $(x,y) \in \mathbb{R}^2$, so that the 
system is area preserving as advertised. 
%x_{n+1} &= x_n \cos(\alpha) - (y_n - x_n^2) \sin(\alpha) \\
%y_{n+1} &= x_n \sin(\alpha) + (y_n - x_n^2) \cos(\alpha)
%\end{align*}
The dynamics of the system are studied for a number of parameter 
values $\alpha$ in the book of \cite{Arr90}. In particular, 
numerical simulations suggest that the system appears to admit 
quasiperiodic 
invariant circles and $K$-periodic systems of such.  
%Writing 
%\begin{align*}
%\begin{pmatrix}
%x_{n+1} \\
%y_{n+1}
%\end{pmatrix} 
%&=
%\begin{pmatrix}
%\cos(\alpha) & -\sin(\alpha)\\
%\sin(\alpha) & \cos(\alpha)
%\end{pmatrix}
%\begin{pmatrix}
%x_n \\
%y_n
%\end{pmatrix}
%+
%x_n^2
%\begin{pmatrix}
%\sin(\alpha)\\
%-\cos(\alpha)
%\end{pmatrix}
%\end{align*}
The map can be seen as a linear rotation matrix 
at the origin, plus a quadratic nonlinearity.
There is one (and only one) fixed point --at the origin and of elliptic stability type --
so that in a small enough neighborhood of the origin we expect the existence of 
large measure sets of KAM tori.  
This expectation is supported by
numerical simulations, as seen for example in Figure \ref{fig:henonPhase}.
for $\alpha = \cos^{-1}(0.24)$.

\begin{figure}[t!]
\centering
\includegraphics[scale=0.2]{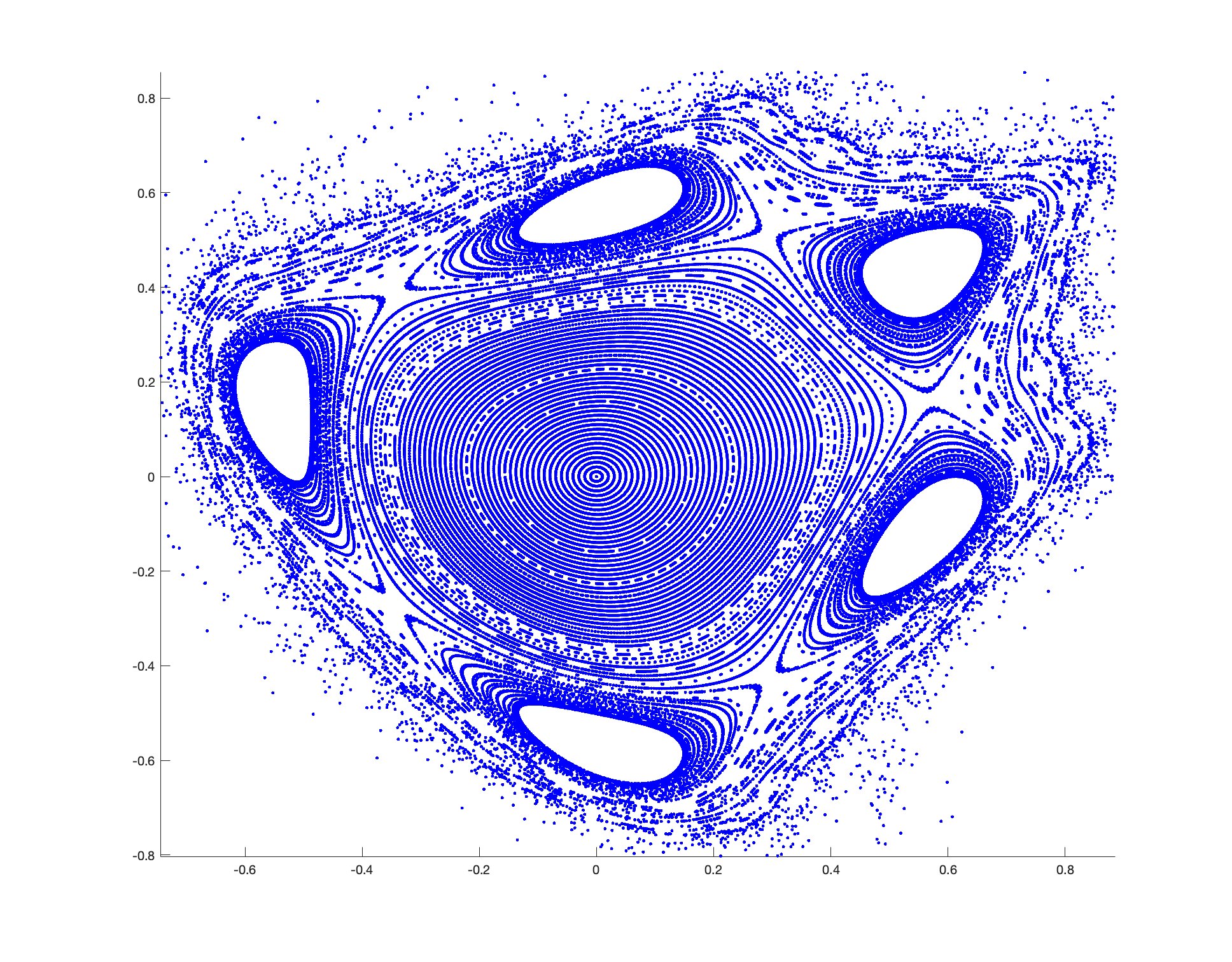} 
\caption{Phase space structure for the area 
preserving Henon map: near the origin the 
dynamics are close to pure rotation and we 
see a large set of invariant circles.  These
get more and more distorted further from the 
origin, and eventually there appears to be 
a $1:5$ resonance which gives rise to a 
family of systems of invariant circles with 
$5$ topological components.  Further from the 
origin the dynamics appears to be chaotic.  
} \label{fig:henonPhase}
\end{figure}

\begin{figure}[t!]
\centering
\includegraphics[scale=0.2]{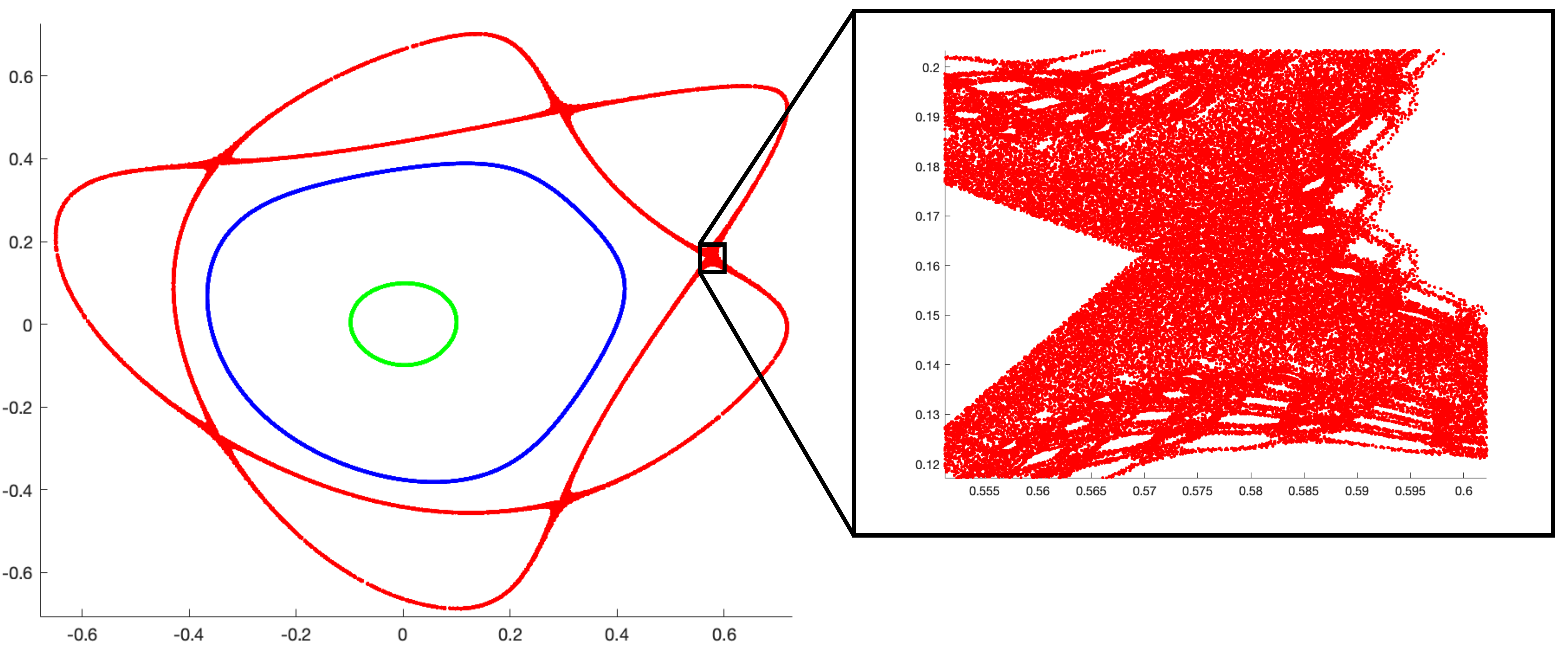} 
\caption{Three orbits: 
one million iterates of the area preserving Henon map
for the three initial conditions $\mathbf{p_0} = (0.1, 0.0)$, 
 $\mathbf{p_0} = (0.4, 0.0)$ and $\mathbf{r_0} = (0.3, 0.44)$.
 (green, blue and red respectively).  Visual inspection suggests 
 that the green and blue curves are diffeomorphic to circles, 
 while zooming in on the point cloud generated by the 
 orbit of $\mathbf{r}_0$ reveals fractal structure and suggests
 chaotic dynamics.  This suggestion is enforced by the quantitative
 data in Table \ref{tab:table1}.
} \label{fig:stepZeroData}
\end{figure}

\begin{table}[t!]
  \begin{center}
       \begin{tabular}{c|c|c|c} % <-- Alignments: 1st column left, 2nd middle and 3rd right, with vertical lines in between
      $M$ & $\rho_M(\mathbf{p_0})$ & $ \rho_M(\mathbf{q}_0)$ & $\rho_M(\mathbf{r}_0)$ \\
      \hline
      100 &  0.211095710088270 & 0.206164038365342 & 0.196863099485937 \\
      500 &0.211095709965501  & 0.206174513248940 &  0.197503558666674 \\
      1000 & 0.211095709965479 & 0.206174514865070 & 0.197628415757003 \\
      5000 &   0.211095709965481   &   0.206174514865715    &   0.199431995293399    \\
      10,000 &  0.211095709965478  & 0.206174514865712  &    0.199737097322017  \\
      50,000 & 0.211095709965486  & 0.206174514865718 &  0.199823145343572 \\
      100,000 & 0.211095709965480 & 0.206174514865710&  0.199984739391916 \\
      110,000 &  0.211095709965478 & 0.206174514865708 & 0.199990822989916 \\
      120,000 &  0.211095709965479 & 0.206174514865704&   0.199994461862213 \\
      150,000 &  0.211095709965479 &  0.206174514865705&   0.199998753698169 \\
      200,000 &  0.211095709965478 & 0.206174514865702 &  0.199999888701773
    \end{tabular}
     \caption{Numerically computed values of the rotation number 
     as a function of $M$ for the 
     three initial conditions $p_0$, $q_0$, and $r_0$.  This is denoted $\rho_M(\cdot)$, 
     where $\cdot$ is one of the initial conditions and $M$ between one hundred and 
     and two hundred thousand.  The computations suggest that the orbits of $p_0$
     and $q_0$ are quasiperiodic, while the rotation number of the orbit of $r_0$
     varies stochastically after the 4th digit.} \label{tab:table1}
  \end{center}
\end{table}

\subsubsection{A worked example: period 1 invariant circle} \label{sec:per1_henon}
We now describe in some detail the computation of a
period one invariant circle for the area preserving Henon map.

\begin{itemize}
\item \textbf{Step 0:} consider the three initial conditions 
$p_0, q_0, r_0 \in \mathbb{R}^2$ given by 
\[
p_0 = 
\left(
\begin{array}{c}
0.1 \\
0.0
\end{array}
\right), \quad \quad 
q_0 = 
\left(
\begin{array}{c}
0.4 \\
0.0
\end{array}
\right), \quad \quad \mbox{and} \quad \quad  
r_0 = 
\left(
\begin{array}{c}
0.3 \\
-0.44
\end{array}
\right).
\]
One million iterates of each initial condition are illustrated in 
Figure \ref{fig:stepZeroData}, with a zoom in on the orbit of 
$r_0$, illustrating that the orbit appears to be chaotic rather than 
quasiperiodic.   This appearance is confirmed by the rotation 
number calculations given in Table \ref{tab:table1}.
Based on these results, and for the rest of the Section, 
we focus on the orbit of $q_0$.   
\item \textbf{Step 1: } Based on the results of step 0, since we
can say with  confidence that the rotation number 
associated with the orbit of $q_0$ has 
\[
\rho \approx \rho_{120,000} = 0.206174514865704,
\]
which is likely correct except possibly in the last decimal place.  
\item \textbf{Step 2: } Using the rotation number computed in the 
last  step, we sample the Fourier coefficients in the higher modes.
We note that with $M = 1,000$ we already appeared to have 
seven correct figures in the rotation number calculation.  
So we will compute Fourier coefficients with an orbit of only this 
length.  Let $p_n = (a_n, b_n)$ denote the $n$-th Fourier vector
and 
\[
\|p_n\| = \max(|a_n|, |b_n|),
\]
with $|\cdot|$ the complex absolute value.  
Sampling the coefficients
for $n = 2, 4, 6, 8, 10, $ we have 
\begin{align*}
\| p_2 \| & = 2.0 \times 10^{-2} \\
\| p_4 \| & =  4.3 \times 10^{-3} \\
\|p_6\| & = 6.5 \times 10^{-4} \\
\|p_8\| & = 4.8 \times 10^{-5} \\
\|p_{10}\| & = 8.4 \times 10^{-6} \\
\end{align*}
Based on the observed decay rate, we guess that 
we should reach machine precision 
at roughly $n = 30$.
Being a little conservative, we take $K = 5$
and truncate to $N = 2^K = 32$ Fourier modes.
(Powers of 2 are desirable if the implementation 
emploies the FFT).
\item \textbf{Step 3: }. We now compute the 
Fourier series for $N_0 = 5$, from 10,000
data points.  This is about 10 percent of the 
modes to be Used in the Newton scheme.  
This leads to a trigonometric polynomial that we refer to 
as 
\[
\tilde K(\theta) = \sum_{n = -5}^5 \left(
\begin{array}{c}
a_n \\
b_n 
\end{array}
\right) e^{2 \pi i n \theta}.
\]
Then initial defect associated with this approximate solution is 
already $\epsilon \leq 9.2 \times 10^{-4}$.  We therefore 
consider this a good initial approximation and 
define $K_0 = \tilde K$.  
\item \textbf{Step 4: }
We run the newton iteration and obtain defects 
\begin{align*}
\epsilon_1 = 2.2 \times 10^{-6} \\
\epsilon_2 = 4.1 \times 10^{-12} \\
\epsilon_3 = 6.1 \times 10^{-13} \\
\epsilon_4 = 6.0 \times 10^{-13}
\end{align*}
and the conjugacy error stagnates.
The Newton scheme executes in $0.062$ seconds.  
%We remark that computing and averaging $120,000$
%iterates of the Henon map to find the rotation number
%already takes $0.048$ seconds.  
Running again from the same initial condition $K_0$
with $N = 64$, the next power of two Fourier modes, 
results in a final conjugacy error of $\epsilon = 1.6 \times 10^{-14}$ and 
takes $0.18$ seconds.  The next power $N = 128$ takes $0.38$ seconds
and results in a conjugacy error of $\epsilon = 6.1 \times 10^{-16}$, 
which is finally on the order of double precision Machine epsilon.  
Truncating at  $N = 256$ Fourier coefficients results in a $0.9$ second 
runtime, and does not improve the conjugacy error.  
Indeed, we see that the initial $N = 32$ calculation was already nearly optimal.  
%12 7
%14 8
%16 9
%18 10
%20 11
%22 12
%24 13
%26 14
%28 15
%30 16
\end{itemize}

We provide a few additional details regarding the 
numerical implementation in this example.  
Let $a = \{a_n \}_{n \in \mathbb{Z}}$ and 
$b = \{b_n \}_{n \in \mathbb{Z}}$ denote the unknown 
Fourier series coefficients for the parameterization $K$.  
Then 
\[
F(K(\theta)) = \sum_{n \in \mathbb{Z}} \left(
\begin{array}{c}
 \cos(\alpha) a_n - \sin(\alpha) b_n + \sin(\alpha)(a*a)_n \\
\sin(\alpha) a_n + \cos(\alpha) b_n - \cos(\alpha) (a*a)_n
\end{array}
\right) e^{2 \pi i n \theta},
\]
where 
\[
(a*a)_n = \sum_{k \in \mathbb{Z}} a_{n-k} a_k,
\]
denotes discrete convolution.  Recalling that 
\[
K(\theta + \rho) = \sum_{n \in \mathbb{Z}} e^{2 \pi i n \rho }
 \left(
\begin{array}{c}
a_n \\
b_n
\end{array}
\right) e^{2 \pi i n \theta},
\]
then the unfolded conjugacy equation $F(K(\theta)) = (1+\beta)K(\theta + \rho)$ 
is satisfied if and only of the Fourier coefficients on the left 
equal the Fourier coefficients on the right, and we require that 
\begin{equation} \label{eq:fourierProject}
 \left(
\begin{array}{c}
 \cos(\alpha) a_n - \sin(\alpha) b_n + \sin(\alpha)(a*a)_n \\
\sin(\alpha) a_n + \cos(\alpha) b_n - \cos(\alpha) (a*a)_n
\end{array}
\right)  = (1+\beta)e^{2 \pi i n \rho }
 \left(
\begin{array}{c}
a_n \\
b_n
\end{array}
\right) \quad \mbox{for } n \in \mathbb{Z}.
\end{equation}
Moreover, noting that the invariant circle given by the data
crosses the $x$-axis we choose the phase condition 
\[
K_2(0) = \sum_{n \in \mathbb{Z}} b_n  = 0.
\]
Truncating at $N$ Fourier modes leads to the system of  $2(2N+1) + 1$ equations
\begin{align*}
b_{-N} + \ldots + b_0 + \ldots b_{N} & = 0 \\ 
 \cos(\alpha) a_{-N} - \sin(\alpha) b_{-N} + \sin(\alpha)(a*a)^N_{-N}  -  (1+\beta)e^{-2 \pi i N \rho } a_{-N} & = 0\\
\sin(\alpha) a_{-N} + \cos(\alpha) b_{-N} - \cos(\alpha) (a*a)^N_{-N} -  (1+\beta)e^{-2 \pi i N \rho } b_{-N} & = 0\\
& \vdots \\
\cos(\alpha) a_{0} - \sin(\alpha) b_{0} + \sin(\alpha)(a*a)^N_{0}  -  (1+\beta) a_0 & = 0\\
\sin(\alpha) a_{0} + \cos(\alpha) b_{0} - \cos(\alpha) (a*a)^N_{0} -  (1+\beta) b_0 & = 0\\
& \vdots  \\
\cos(\alpha) a_{N} - \sin(\alpha) b_{N} + \sin(\alpha)(a*a)^N_{N}  -  (1+\beta)e^{2 \pi i N \rho } a_N& = 0\\
\sin(\alpha) a_{N} + \cos(\alpha) b_{N} - \cos(\alpha) (a*a)^N_{N} -  (1+\beta)e^{2 \pi i N \rho } b_N & = 0\\
\end{align*}
in the $2(2N+1) + 1$ unknowns 
$\beta, a_{-N}, b_{-N}, \ldots, a_0, b_0, \ldots, a_{N}, b_N$.
Here 
\[
(a * a)_n^N = \sum_{\stackrel{k_1 + k_2 = n}{-N \leq k_1, k_2 \leq N}} a_{k_1} a_{k_2},
\]
is the truncated discrete convolution.  
Newton's method is used to solve this system.  

A higher level representation is obtained as follows.
Let $a = \{a_n\}_{n \in \mathbb{Z}}$ and $b = \{b_n \}_{n \in \mathbb{Z}}$
denote the unknown sequences of Fourier coefficients and define the 
``diagonal'' linear operator $R_\rho$ on an infinite sequence by 
\begin{equation} \label{eq:def_Rrho_operator}
(R_\rho a )_n = e^{2 \pi i n \rho} a_n. 
\end{equation}
Inspired by the conditions given in Equation \eqref{eq:fourierProject}, 
we define the mapping 
\[
\mathcal{F}(\beta, a,b) = \left(
\begin{array}{c}
\sum_{n \in \mathbb{Z}} b_n  \\
\cos(\alpha) a - \sin(\alpha) b + \sin(\alpha) a*a - (1+\beta) R_\rho a \\
\sin(\alpha) a + \cos(\alpha) b - \cos(\alpha) a*a - (1+\beta) R_\rho b
\end{array}
\right), 
\]
and seek a zero of $F$.  Note that, for numbers $\delta$ and 
infinite sequences $u, v$, 
we see that the action of the (formal)
Frechet derivative on $(\delta, u, v)$ is given by 
\[
D\mathcal{F}(\beta, a, b)(\delta, u, v) = 
\left(
\begin{array}{c}
\sum_{n \in \mathbb{Z}} v_n \\
\cos(\alpha) u - \sin(\alpha)v + 2 \sin(\alpha) a * u - (1+\beta) R_\rho u - R_\rho a \\
\sin(\alpha) u + \cos(\alpha) v - 2\cos(\alpha) a * u - (1 + \beta )R_\rho v - R_\rho b 
\end{array}
\right).
\]
A more useful is the following expression for the 
derivative as a ``matrix of operators.''

Let $\mathbf{0}$ denote the zero Fourier sequence and $\mathbf{1}$ the 
sequence of ones.  Moreover, let $\mathbf{R}_\rho$ denote the bi-infinite 
diagonal matrix with $e^{2 \pi i n \rho}$ on the diagonal entries, and 
let $\mathbf{S}_\alpha, \mathbf{C}_\alpha$ denote the bi-infinite diagonal 
matrices with $\sin(\alpha)$ and $\cos(\alpha)$ on their diagonals respectively.
Finally, let $\mathbf{A}$ denote the (dense) bi-infinite matrix defined by the linear mapping 
\[
\mathbf{A} h = a * h.
\]
The matrix for $A$ is easily worked out by considering it's action on 
the basis for bi-infinite sequence space given by sequences with on one non-zero 
entry.  The classical result is that $A$ is a Topoletz matrix for the bi-infinite sequence $a$.  
Then the derivative can be represented as 
\[
D\mathcal{F}(\beta, a, b) = 
\left(
\begin{array}{ccc}
0 & \mathbf{0} & \mathbf{1} \\
-\mathbf{R}_\rho a 
& \mathbf{C}_\alpha+ 2 \mathbf{C}_\alpha \mathbf{A} - (1+\beta) \mathbf{R}_\rho 
& - \mathbf{S}_\alpha \\
- \mathbf{R}_\rho b & \mathbf{S}_\alpha- 
2 \mathbf{C}_\alpha \mathbf{A} & \mathbf{C}_\alpha - (1+\beta) \mathbf{R}_\rho
\end{array}
\right).
\]
Truncating the mapping $\mathcal{F}$ and its derivative given above leads
to a numerical implementation of the Newton scheme.

\subsection{Period $K$ circles in H\'enon} \label{sec:per5}

Another apparent feature of the phase space readily visible in 
Figure \ref{fig:henonPhase}, is what looks like a family of period 5
invariant circles.  After visual inspection of the figure, 
we plot a trajectory using
$(x_0, y_0) = (0.5, 0)$ as our seed point, and observe that 
after 1000 iterates, the orbit appear to fill out the five circles shown 
in the left frame of Figure \ref{fig:per5}. 
We remark that each iterate jumps from one circle to the next circle to 
its left (counter clockwise rotation).  

\begin{figure}[t!]
\begin{centering}
\includegraphics[width=.46\textwidth]{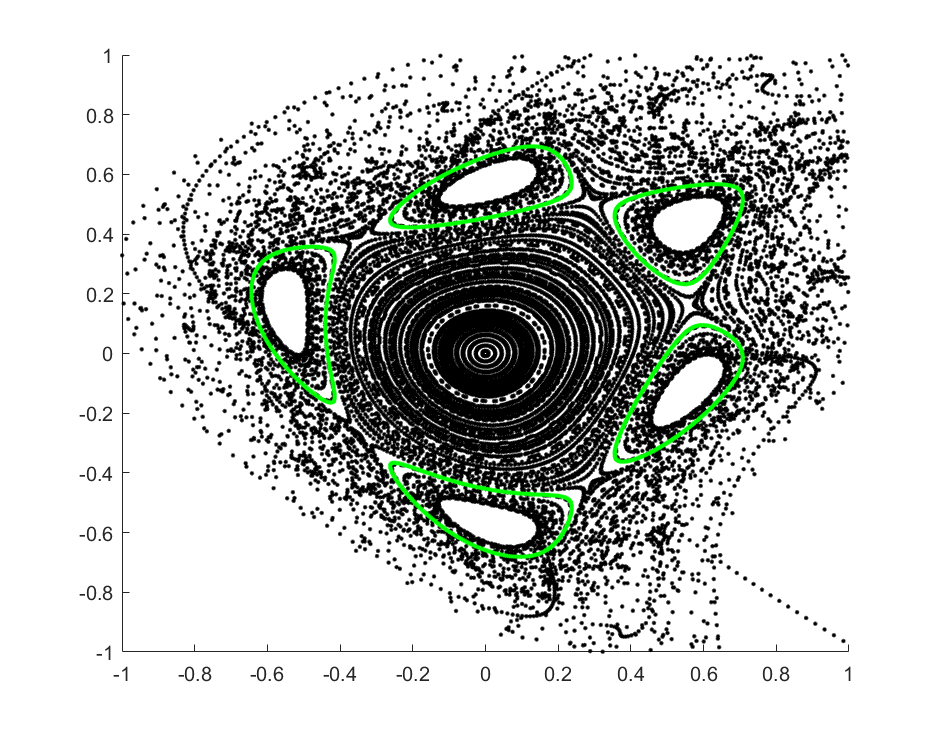} 
\includegraphics[width=.46\textwidth]{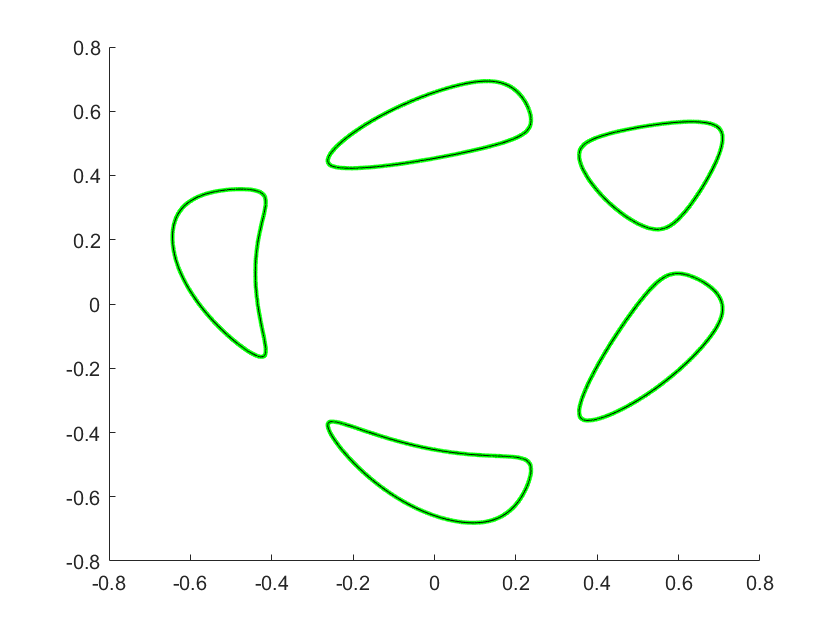} 
\end{centering}
\caption{\textbf{A period five system of quasiperiodic invariant circles:}
left frame illustrates a numerical simulation of the area preserving system, 
and an orbit which appears to lie on the period 5 system of invariant circles.
The right frame illustrates the image of the five Fourier series parameterizing 
the quasiperiodic system.  
} \label{fig:per5}
\end{figure}

\subsubsection{Multiple shooting invariance equations} \label{sec:multShoot_per5ex}
The idea is to follow the steps proposed in Section \ref{sec:numericalRecipe}, 
with a few small modifications. 
After guessing a point on the period $5$ system, we compute an 
orbit segment for the fifth iterate of $F$, denoted $F^5$, and compute the 
rotation number for the composition.  This means that if we 
desire an  orbit segment of length $M$, we have to iterate $F$
$5M$ times.  Since all $5$ circles have the same rotation number, 
this only has to be done once.  Using the Birkhoff averages 
with an orbit segment of length $5 \times 9,000$ leads to 
$\rho = 0.190669478955264$ which has stabilized numerically to the last 
digit.  

Now let $K_1, K_2, K_3, K_4, K_5 \colon \mathbb{R} \to \mathbb{R}^2$ denote the 
desired parameterizations for the five component circles of the system.
We use the weighted Birkhoff averages to compute (roughly) the decay 
rate of these Fourier series (to guess that the optimal truncation order is 
around $N = 200$) and to approximate
the first few Fourier coefficients in each case.  
Again, for this work 
we deal only with (shorter) orbit segments for the composition map $F^5$.
We stress that this just requires computing a long enough orbit for $F$
and then neglecting all but every fifth point on the orbit.

Now, when it comes to the Newton method we work with multiple shooting 
system of equations, so that the nonlinearity is still only quadratic
(note that $F^5$ is a polynomial map of degree $2^{5} = 32$).
Keeping in force the notation from Section \ref{sec:per1_henon}, we 
define the mapping 
\[
\mathcal{F}_5(\beta, a^1, b^1, a^2, b^2, a^3, b^3, a^4, b^4, a^5, b^5) =
\]
\[
\left(
\begin{array}{c}
\sum_{n \in \mathbb{Z}} b^1_n \\
\cos(\alpha) a^1 - \sin(\alpha) b^1 + \sin(\alpha) a^1*a^1 -  R_\rho a^2 \\
\sin(\alpha) a^1 + \cos(\alpha) b^1 - \cos(\alpha) a^1*a^1 -  R_\rho b^2 \\
\cos(\alpha) a^2 - \sin(\alpha) b^2 + \sin(\alpha) a^2*a^2 -  R_\rho a^3 \\
\sin(\alpha) a^2 + \cos(\alpha) b^2 - \cos(\alpha) a^2*a^2 - R_\rho b^3 \\
\cos(\alpha) a^3 - \sin(\alpha) b^3 + \sin(\alpha) a^3*a^3 -  R_\rho a^4 \\
\sin(\alpha) a^3+ \cos(\alpha) b^3 - \cos(\alpha) a^3*a^3 -  R_\rho b^4 \\
\cos(\alpha) a^4 - \sin(\alpha) b^4 + \sin(\alpha) a^4*a^4 -  R_\rho a^5 \\
\sin(\alpha) a^4 + \cos(\alpha) b^4 - \cos(\alpha) a^4*a^4 -  R_\rho b^5 \\
\cos(\alpha) a^5 - \sin(\alpha) b^5 + \sin(\alpha) a^5*a^5 - (1+\beta) R_\rho a^1 \\
\sin(\alpha) a^5 + \cos(\alpha) b^5 - \cos(\alpha) a^5*a^5 - (1+\beta) R_\rho b^1 \\
\end{array}
\right),
\]
and have that if $(\beta, a_1, b_1, a_2, b_2, a_3, b_3, a_4, b_4, a_5, b_5)$ is a 
zero of $\mathcal{F}_5$, then $\beta = 0$ the $a's$ and $b's$ are the Fourier 
coefficient sequences of the parameterizations $K_1, K_2, K_3, K_4$, and $K_5$
for the system of invariant circles.  
Note that while the map has more components, the nonlinearity is still only as 
complicated as that of $F$. In this case quadratic.  
The derivative of $\mathcal{F}$ is easily computed. Truncating 
the map and its derivative leads to the numerical implementation of the 
Newton method.  Note that all the operations and linear operators 
are as in the case of a period one circle.  Only the number of components
and the coupling is different.  
After implementing these adjustments we are able to compute 
the parameterizations to machine precision as in the earlier example. 
The resulting Fourier series are plotted in the right Frame of Figure \ref{fig:per5}.

\begin{figure}[t!]
\begin{centering}
\includegraphics[width=.3\textwidth]{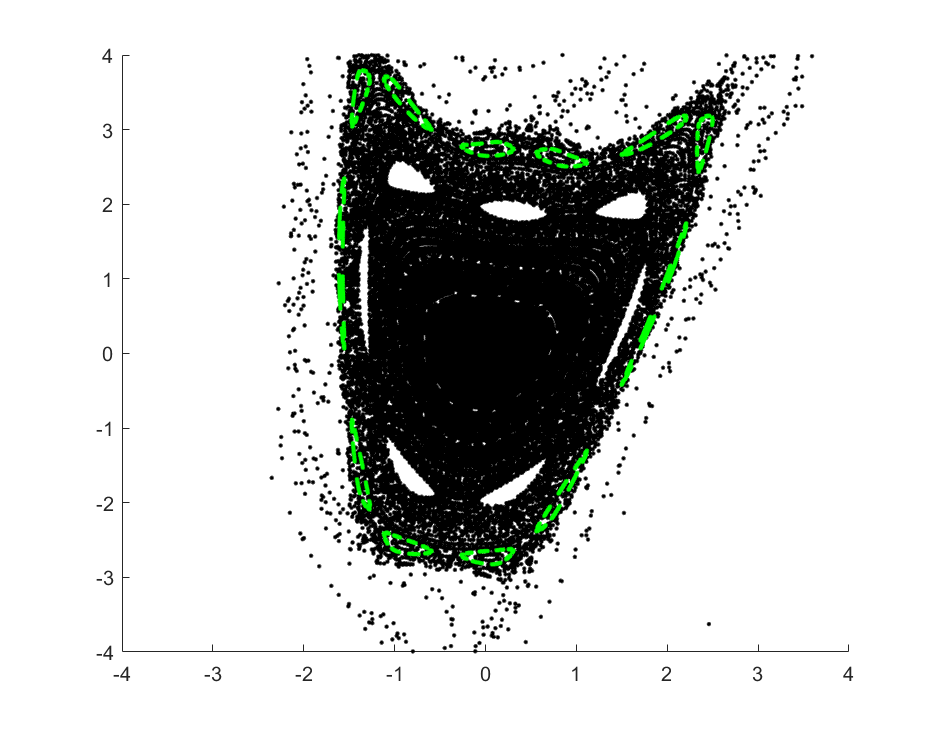} 
\includegraphics[width=.3\textwidth]{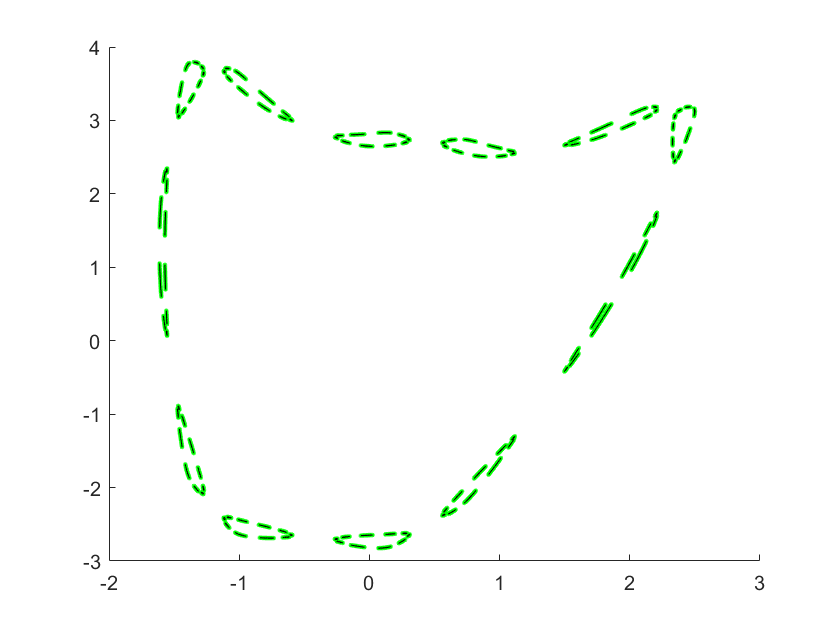} 
\includegraphics[width=.3\textwidth]{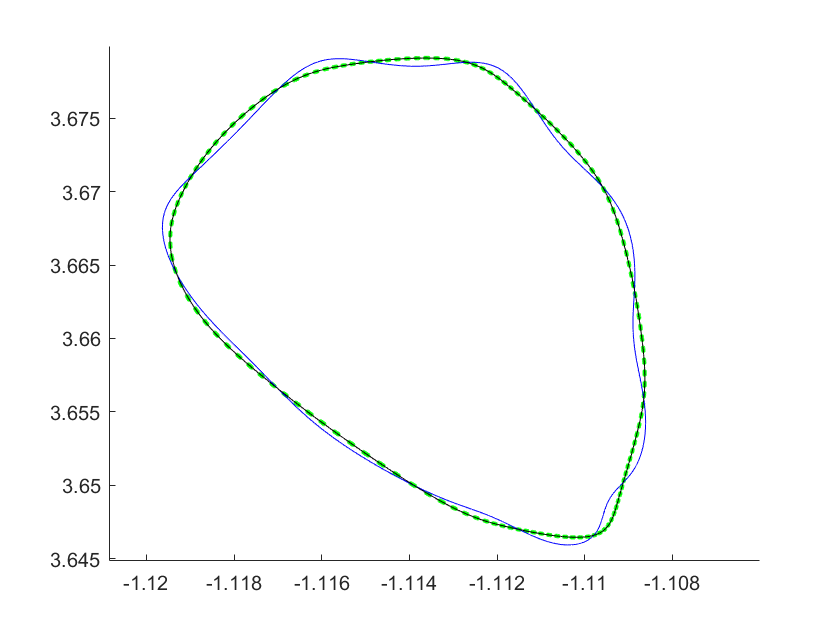}
\end{centering}
\caption{\textbf{A period 120 system of quasiperiodic invariant circles:}
left frame is the phase space simulation and orbit data.  Middle frame 
illustrated the images of Fourier parameterizations.  Right frame 
is a close up of the initial and final parameterizations of a single component 
circle.  
} \label{fig:per120}
\end{figure}

The phase space for the area preserving Henon when $\alpha = \arccos(-0.95)$
is illustrated in the left frame of Figure \ref{fig:per120}, and there is the 
suggestion of even longer systems of invariant circles.  For example, 
repeating the procedure discussed in the proceeding section
using the initial condition $(x_0,y_0) =  (0,-2.65)$ leads to the period 120 system of quasiperiodic 
invariant circles illustrated in the left frame of Figure \ref{fig:per120}.  
The Fourier mapping $\mathcal{F}_{120}$ generalizes from $\mathcal{F}_5$ 
in the obvious way.  Fortunately, the individual circles are not terribly 
complicated harmonically, and $15$ modes per circle appears to be enough to 
approximate the Fourier expansions well. 
Again, the Newton method converges with and 
we obtain the parameterizations $K_1, \ldots, K_{120}$ whose
images are illustrated in the center frame of Figure \ref{fig:per120}.
The right frame illustrates the initial and final parameterizations at a 
zoom in on one of the 120 components.

\begin{figure}[t!]
\centering
\includegraphics[scale=0.6]{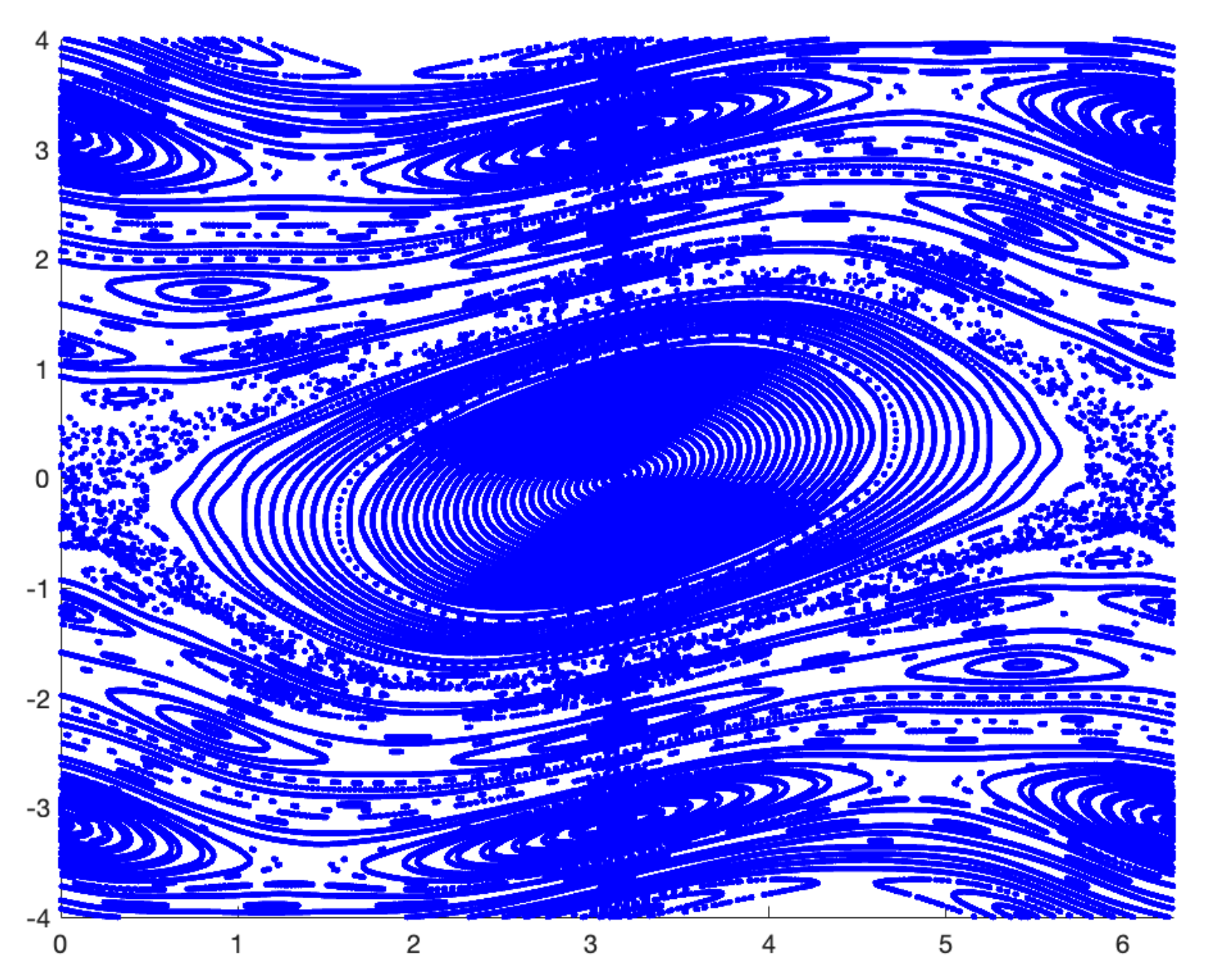} 
\caption{Phase space structure for the standard map
with $\alpha = \pi/4$:
in the simulation illustrated in this figure,
 we have taken the first component modulo 
$2 \pi$ so the line at $x = 0$ is identified with the line at 
$x = 2 pi$.  This has the effect of making the 
phase space into a cylinder, and there are
primary invariant tori which appear as smooth 
curves running from left to right.  In the present work
however, we study the secondary tori associated with the 
elliptic fixed point at $(\pi, 0)$.  These tori are visible
wether or not we compute modulo $2pi$.  
} \label{fig:stdPhase}
\end{figure}

\subsection{Computations for the Standard Map} \label{sec:standardMap}

For an example of a map with non-polynomial nonlinearity, 
consider the Standard Map of \cite{Chi79}. 
Since we are interested in secondary (contractable) invariant 
tori, we treat the map as a diffeomorphism 
$F: \mathbb{R}^2 \to \mathbb{R}^2$ given by the formula,
\begin{equation} \label{eq:defStdMap}
F(x,y) = \left(
\begin{array}{c}
x + y + \alpha \sin(x) \\
y + \alpha \sin(x)
\end{array}
\right).
\end{equation}
That is, we only take results modulo $2\pi$ in the first component
of the map to produce graphical results.

One subtle question is weather to consider the phase space as
$\mathbb{R}^2$ or $\mathbb{T} \times \mathbb{R}$.  In the 
later case, we take the first component of $F$ defined in Equation 
\eqref{eq:defStdMap} modulo $2 \pi$, forcing a periodicity in $x$.
A phase space simulation is illustrated in Figure \ref{fig:stdPhase} 
for a large value of $\alpha$.  Note that while there are many 
primary invariant circles (curves which wind around the cylinder
in a non-trivial) visible in this simulation, the main feature in 
is that resonance zone near the elliptic fixed point at
$(\pi, 0)$.  We remark that the secondary invariant circles
about this fixed point (which are contractible on the cylinder)
remain invariant even if we take the phase space to be 
$\mathbb{R}^2$.  The non-contractible invariant 
circles are the focus of this article, as they
are in some sense more difficult to compute.
This is because they cannot be treated using the skew 
product formulation, where a non-contractible 
invariant circle is written as the graph of a periodic function
(1d computations).

%\begin{align*}
%q_{n+1} &= (q_n + p_{n+1}) \mod 2\pi \\
%p_{n+1} &= p_n + \alpha \sin(q_n) 
%\end{align*}

%
%Taking the sine of a Fourier series is not incredibly difficult, however, taking it with our desired accuracy is. 
%Hence, we use automatic differentiation \cite{MR3545977}, or quadratic reduction (CITATION)
% where we add equations to the system to guarantee an accurate approximation of sine. 
%
%In this case, we may think of the system as a perturbation of an integrable system as explained in \cite{DLL01}. 
%However, in the literature the concern is with finding quasiperiodic circles that survive the perturbation as in the 
%so called golden curve \cite{mamotreto}, see Chapter 4. 
%We focus exclusively on periodic circles that arise because of the non-linearity. 
%The period $n > 1$ are mentioned in the literature as oscillational invariant circles \cite{MR4104977}
% that exist in the island chains that make up a basic component of the dynamics in perturbed Hamiltonian systems. 
%

\subsubsection{Period 1 Standard Map}

Taking $\alpha = \pi/4$, we consider the orbit of the 
point $P = (\pi, 1)$.  Simulations suggest that the 
orbit is dense in an invariant circle, and we 
proceed as in the example of the period one 
computation for the area preserving Henon map 
discussed in Section  \ref{sec:per1_henon},
implementing the numerical recipe discussed in Section 
\ref{sec:numericalRecipe}.  
Computing with $12,000$ data points, 
we find the rotation number to be
$\rho \approx 0.871221766629878$. (Here there is 
a difference of 5.551115e-16 compared to the rotation number computed with
 11000 points, and we trust roughly 15 if the 16 computed digits). 
 
Truncating the parameterization to 
$N_0 = 20$ Fourier modes, computing with only 
a length 100 orbit segment
yields an approximate parameterization with initial defect of roughly $0.1288$.
Beginning with this as an initial approximation, the Newton method 
(truncated at $N = 50$ modes) converges
to the solution illustrated in Figure \ref{fig:per1_stdMap}.
The conjugacy error of the final approximation is on the order of machine epsilon.

%Applying the newton scheme, truncating to 50 modes, we achieve a defect of 8.666519e-08. 
%The decay of the coefficients is much slower than in H\'enon, but this is possibly attributable to the wrapping by $2\pi$ and the autodifferentiation, in addition to the stronger non-linearity. 
%The Sobolev norms of the solution parameterization, however, appear to be in line with the H\'enon case.
%

\begin{figure}[t!]
\begin{centering}
\includegraphics[width=.9\textwidth]{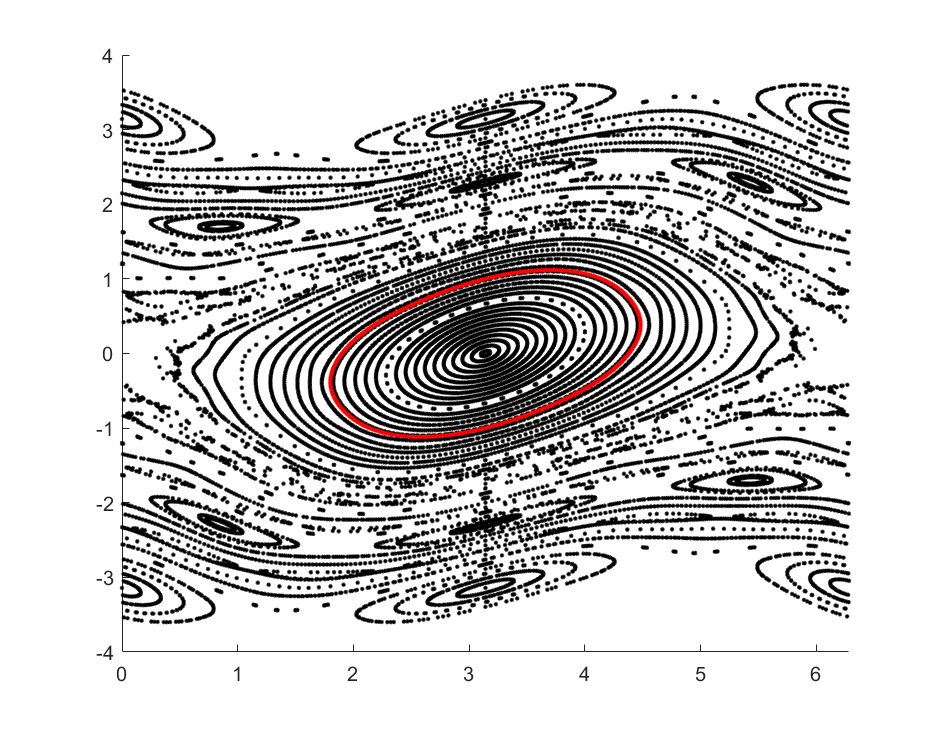} 
\end{centering}
\caption{\textbf{A period 1 quasiperiodic invariant circle for the standard map: }
phase space simulation and the converged parameterization of a
quasiperiodic invariant circle with rotation number 
$\rho \approx 0.871221766629878$ compute using $50$ Fourier modes.  
} \label{fig:per1_stdMap}
\end{figure}

Taking $\alpha = \pi / 2$ and initial points $P = [1.85, 0.565]$, 
and $P = [4.8155, 0.5]$, 
we compute the Fourier parameterizations of
period $6$ and period $24$ quasiperiodic systems
invariant circles using the ideas described in Section \ref{sec:per5}.
These results are illustrated in Figures 
\ref{fig:per6_stdMap} and \ref{fig:per24_stdMap}, and show that 
the multiple shooting parameterization method works 
also for non-polynomial nonlinearities.

\begin{figure}[t!]
\begin{centering}
\includegraphics[width=.45\textwidth]{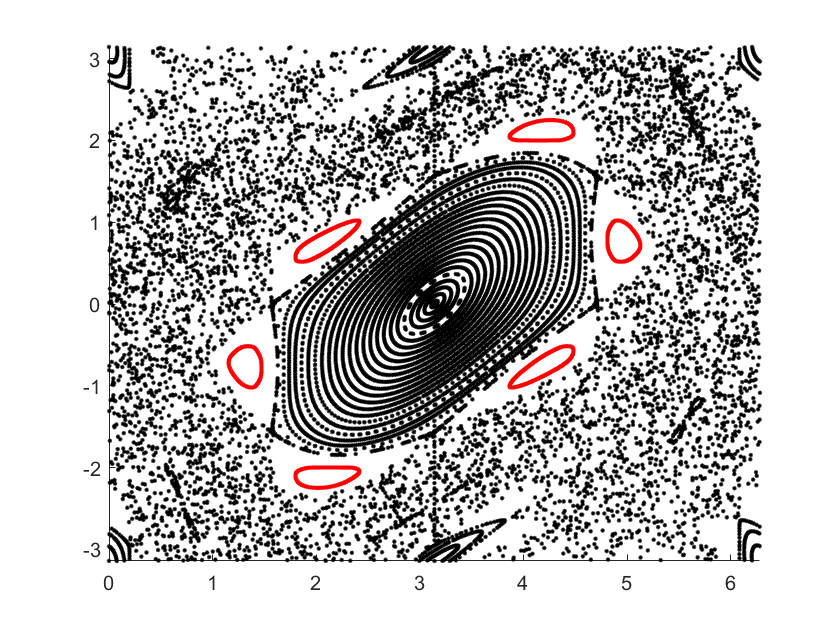}
\includegraphics[width=.45\textwidth]{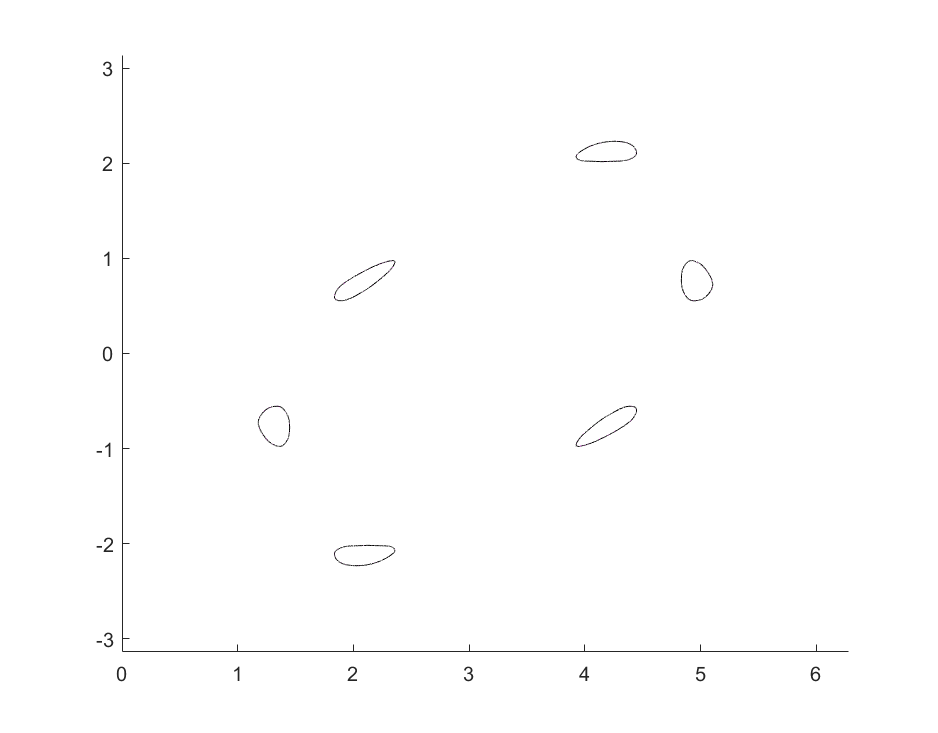}
\end{centering}
\caption{\mbox{A period 6 quasiperiodic system of invariant circles for the standard map: }
left figure is a phase space simulation and orbit segment data.  Right frame 
illustrates the image of the converged Fourier approximation.
} \label{fig:per6_stdMap}
\end{figure}

%\includegraphics[width=.45\textwidth]{standard_pi-2_space.png}

%
%Next, we look for more features nearby the period 6 trajectory. 
%We consider $P = [4.8155, 0.5]$ and observe that the trajectory appears to be of period 24. 
%Using 25 modes, the Newton method produces a parameterization with a sequence space error of 8.881784e-16. 
%

\begin{figure}[t!]
\begin{centering}
\includegraphics[width=.45\textwidth]{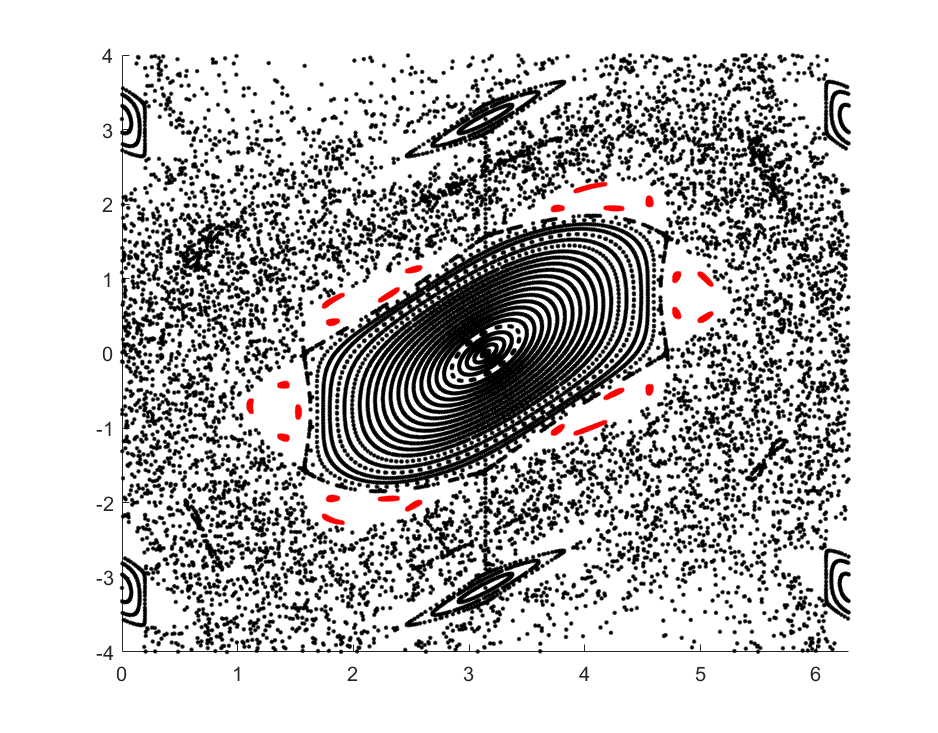} 
\includegraphics[width=.45\textwidth]{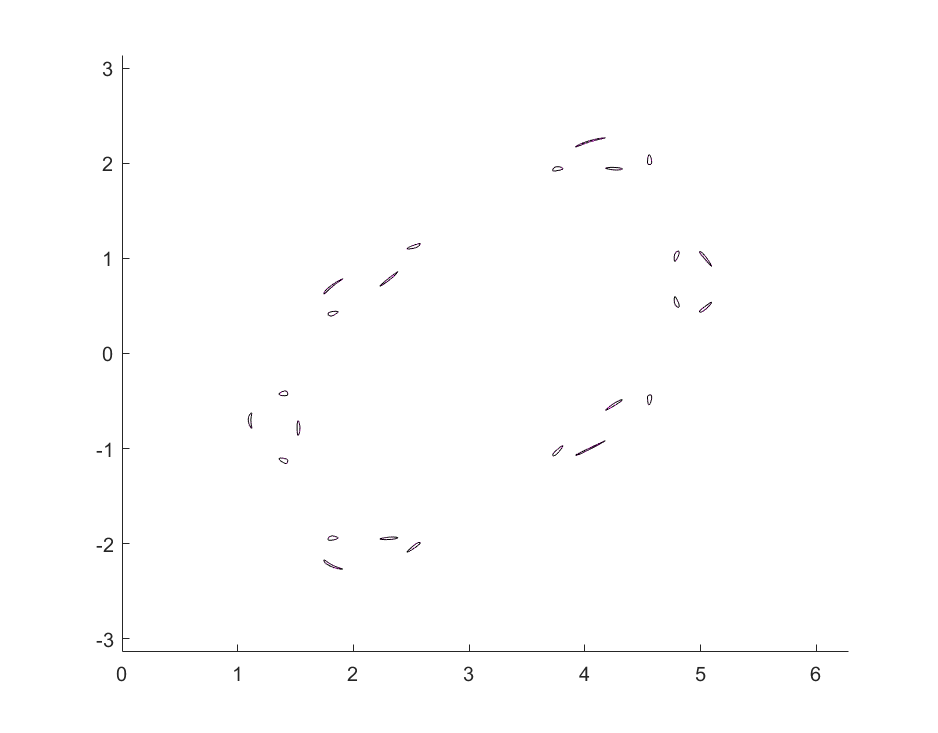} 
\end{centering}
\caption{\textbf{A period 24 quasiperiodic system of invariant circles for the standard map:}
same left and right as the previous figure.  
} \label{fig:per24_stdMap}
\end{figure}

We cap off this overview of the higher period standard map examples with the observation that the method described produces robust results with small sequence space error. 
However, the conjugacy error is not so easily controlled, while small, it has so far proved intractable to make arbitrarily so.

\subsubsection{Polynomial embedding of non-polynomial nonlinearities} \label{sec:polynonmialEmbedding}
In this section we include a few remarks about the implementation details for the 
nonlinearity in the standard map.  Indeed, suppose that $f$ is a period-$1$ function given by 
\[
f(\theta) = \sum_{n \in \mathbb{Z}} a_n e^{2 \pi i \theta}, 
\]
and that we want to compute
the Fourier coefficients of the composition 
\[
\sin(f(\theta)) = \sum_{n \in \mathbb{Z}} a_n e^{2 \pi i \theta}.
\]
One approach (perhaps the most natural) is to employ the FFT, and
if  $f$ is an arbitrary band-limited function 
and $g$ is smooth, then this in general provides 
the best known method for computing the Fourier coefficients
of $g \circ f$. In the present setting however, the functions being 
composed have additional structure.  They are
solutions of certain polynomial functional equations
and, by appending these equations to the parameterization method, we obtain 
a new functional equation whose nonlinearity is only polynomial (in fact quadratic).  
This avoids the overhead of implementing the FFT, and more importantly overcomes the 
``numerical stagnation'' of the coefficient decay of the composition 
at 10 or 20 multiples of Machine precision -- as is often observed 
when interpolation based methods for evaluating 
spectral coefficients are used.

The technique described here is 
give different names in different communities, for example  
automatic differentiation \cite{mamotreto,MR3545977},
polynomial embedding \cite{MR4413373,MR4292532}, and 
quadratic recast \cite{MR3973675,Cochelin2008AHP} to name only a few.
We refer to \cite{MR2194171} and also to Chapter 4.7 of 
\cite{MR633878} for a more thorough discussion of the 
history of these ideas, going back to the 19th Century.  
We explain the idea for a period one invariant circle of the standard map.

Consider an invariant circle parameterized by 
\[
K(\theta) = \left(
\begin{array}{c}
K_1(\theta) \\
K_2(\theta)
\end{array}
\right) = 
\sum_{n \in \mathbb{Z}} \left(
\begin{array}{c}
a_n \\
b_n 
\end{array}
\right) e^{2 \pi i \theta}
\]
which passes through the $x$-axis when 
$\theta = 0$.
Then an appropriate phase condition is $K_2(0) = 0$, and we
seek a zero of the operator 
\begin{equation} \label{eq:stdMapFE}
\Psi(a,b, \beta) = \left(
\begin{array}{c} 
\sum_{n \in \mathbb{Z}} b_n \\ 
a + b + \alpha S(a) - (1+\beta) R_\rho a \\ 
b + \alpha S(a) - (1 + \beta) R_{\rho} b
\end{array}
\right).
\end{equation}
Here $a = \{a_n\}, b = \{b_n\}$ are the Fourier coefficient sequences, 
$R_\rho$ is the diagonal operator defined in coefficient space 
in Equation \eqref{eq:def_Rrho_operator}, and $S$ denotes the map 
in coefficient space from $a$ to the Fourier coefficients of $\sin(K_1(\theta))$.
Let $C$ denote the complimentary function which maps the Fourier coefficient 
sequence $a$ to the Fourier coefficients of the function $\cos(K_1(\theta))$.  

We write $S = S(K_1(\theta))$ and $C = C(K_1(\theta))$ to denote the values 
of $S$ and $C$ at $K_1$.  
Note that $S,C$ have 
\[
\frac{d}{d \theta} S(K_1(\theta)) =  S'(K_1(\theta)) K_1'(\theta) = C K_1'(\theta),
\]
and 
\[
\frac{d}{d \theta} C(K_1(\theta)) = C'(K_1(\theta)) K_1'(\theta) = - S K_1'(\theta),
\]
with initial conditions 
\[
S(0) = \sin\left(\sum_{n\in \mathbb{Z}} a_n \right), \quad \quad \mbox{and}\quad \quad 
C(0) = \cos\left(\sum_{n \in \mathbb{Z}} a_n \right).
\]
Let $s = \{s_n\}$, $c = \{c_n\}$ denote the Fourier coefficient sequences of $S$, and $C$, and 
define the diagonal differentiation operator
\[
D(a)_n = 2 \pi i n a_n,  \quad \quad \quad n \in \mathbb{Z}.
\]  

Now suppose that $a,b,c,s,\beta, \gamma, \omega$ is a zero of the operator 
\begin{equation} \label{eq:stdMap_quadratic}
\Psi(a,b,c,s,\beta, \gamma, \omega) = 
\left(
\begin{array}{c}
\sum_{n \in \mathbb{Z}} b_n\\
\sum_{n \in \mathbb{Z}} s_n  - \sin\left(
\sum_{n \in \mathbb{Z}} a_n
\right) \\
\sum_{n \in \mathbb{Z}} c_n  - \cos\left(
\sum_{n \in \mathbb{Z}} a_n
\right) \\
a + b + \alpha s - (1+\beta) R_\rho a \\ 
b + \alpha s - (1 + \beta) R_{\rho} b \\
Ds - c *  D a - \gamma s + \omega c \\
Dc + c * Da - \gamma c - \omega s
\end{array}
\right).
\end{equation}
It can be shown (using an argument similar to the proof of Lemma \ref{lem:unfolding})
that $\gamma$ and $\omega$ are unfolding parameters for the differential equations.
That is, if the initial conditions are satisfied 
(i.e. the second and third components of $\Psi$ are zero)
and if the sixth and seventh components are zero, then 
$\gamma = \omega = 0$.  In this case $s$, $c$ are the Fourier coefficient
sequences of $\sin(a), \cos(b)$ respectively.  It follows that $a, b, \beta$  solve 
Equation \eqref{eq:stdMapFE}.  It then follows from the 
area preserving property of the standard map that $\beta = 0$.
Then  $a, b$ are the Fourier coefficients of a parameterization of an invariant 
circle conjugate to irrational rotation $\rho$.   
We stress that $R_\rho$ and $D$ are diagonal linear operators in Fourier space
and that $*$ is just the discrete convolution.  The 
 operator defined in Equation
\eqref{eq:stdMap_quadratic} is then linear except in the last two components where 
there appear quadratic nonlinearities. This is like a kind of ``multiple shooting'' for 
unwrapping compositions, and it is easily extended to the functional equations
for periodic systems of invariant circles.

\section{Numerical continuation (discrete) for families of periodic  invariant circles}
\label{sec:continuation}
While the recipe given in Section \ref{sec:numericalRecipe} is non-perturbative, requiring only 
finite data sampled form an invariant circle, it is well known (from KAM theory) 
that quasiperiodic invariant circles for area preserving maps typically appear in 
Cantor sets of large measure. We refer the reader to any of the classic books/lecture notes of 
\cite{DLL01,Arr90,MR1995704,MR1979140,MR1792240}, and to their bibliographies
for much more complete references.
We only note that if $\rho$ is irrational (say Diophantine), then $h \rho$ is irrational
(and likely Diophantine) for rational not too small $h$.  
Suppose now that $\Gamma$ is a quasiperiodic invariant circle 
with rotation number $\rho$ and that $h$ is a rational number near $1$. 
Heuristically speaking, it is probable that there 
exists a nearby quasiperiodic invariant circle $\bar \Gamma$ with irrational 
rotation number $h \rho$. 

This suggests that, having found a parameterized invariant circle using the 
method of Section \ref{sec:numericalRecipe},
we perform a kind of (discrete) continuation in the parameter $\rho$.  
That is, suppose that $\rho_0 \in [0,1]$ and $K_0 \colon \mathbb{R} \to \mathbb{R}^2$
is one periodic with 
\[
F(K_0(\theta)) = K_0(\theta + \rho_0), \quad \quad \quad \theta \in [0,1].
\]
Then we take $\rho_1 = h \rho_0$ (with $h$ close to one) and use 
$K_0$ as the initial condition for a Newton method solving the equation 
\[
F(K(\theta)) = K(\theta + \rho_1).
\]
The equation just stated is of course solved using the Newton scheme 
described in Section \ref{sec:parmMethodInvCirc}.  Indeed, it is very likely that 
the new calculation can be performed with exactly the same phase condition.  
The continuation schemes applies also to period-$K$ systems of quasiperiodic 
invariant circles.

This kind of discrete continuation (we use the word ``discrete'' to stress that the 
family of invariant circles does not vary continuously with $\rho$) has been used 
many times in the past.  Indeed in \cite{MR2672635} the authors
show that, for planar symplectic maps, a precursor to ``breakdown'' or disappearance of a 
family of KAM tori is the blow-up of certain Sobolev norms associated with the 
parameterization.  

More precisely, a typical invariant circle in the family is actually analytic, so that there is 
a $\nu > 1$ so that 
\[
\| K \|_0 = \sum_{n \in \mathbb{Z}} \max(|a_n|, |b_n|) \nu^|n| < \infty.
\]
Defining the Sobolev norms 
\[
\| K\|_d^2 = \sum_{n \in \mathbb{Z}} (1 + d^2)^{|n|} \max(|a_n|^2, |b_n|^2)  < \infty,
\]
the main result of  \cite{MR2672635} can be summarized by saying that 
if $\Gamma_\infty$ is ``the last'' invariant torus in the Cantor family (the torus at which 
the family breaks down) then there is a $d$ so that the $d$-th Sobolev norm of 
the parameterization of $\Gamma_\infty$ is infinite.  This suggests choosing a 
$K \in \mathbb{N}$ (perhaps $K = 10$) and monitoring 
the Sobolev norms of the Fourier series coefficients for $1 \leq d \leq K$
during the numerical continuation.  If one begins to blow up we conclude that 
we are near the breakdown.  This can be used as an automatic stopping procedure 
for the continuation.

\begin{figure}[t!]
\begin{centering}
\includegraphics[width=.45\textwidth]{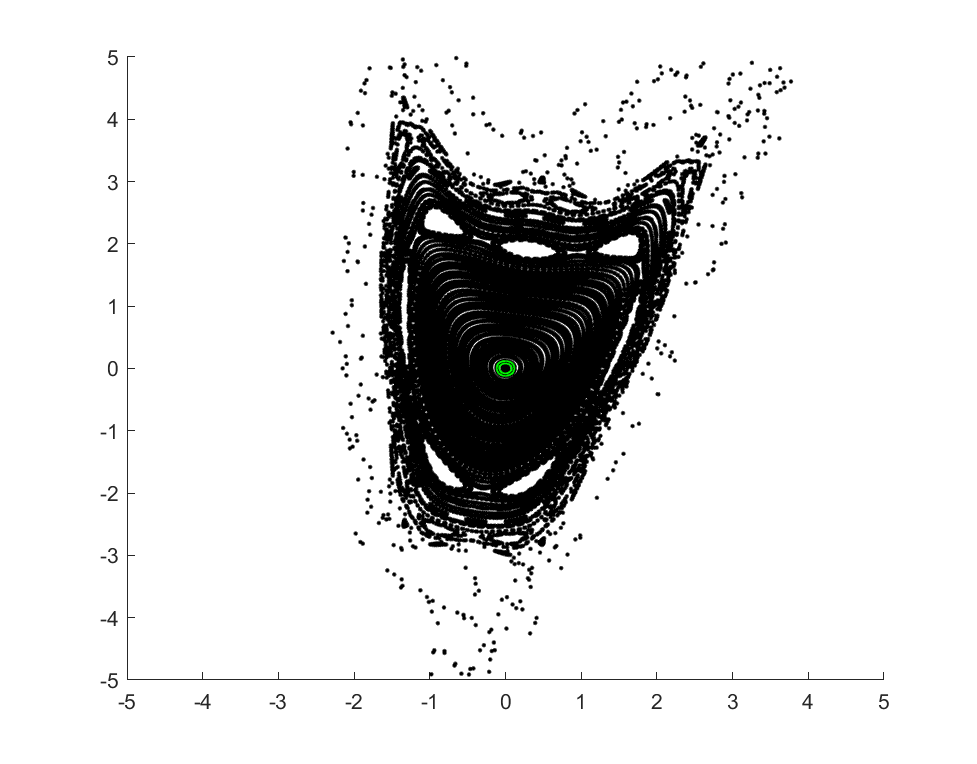}
\includegraphics[width=.45\textwidth]{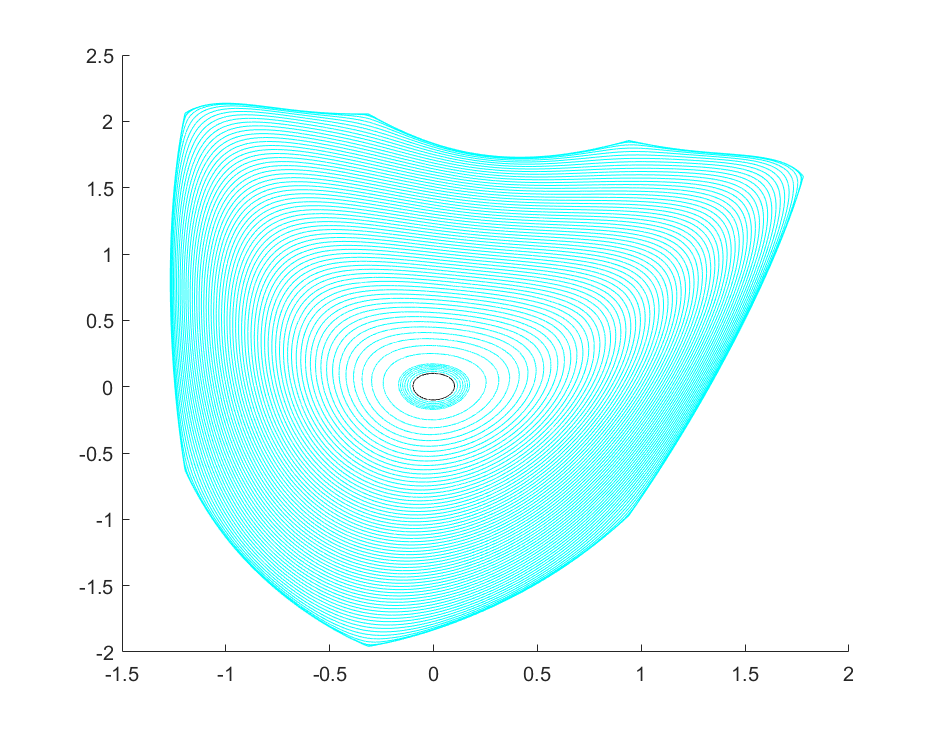}
\end{centering}
\caption{\textbf{Discrete continuation of a period one family for the area perserving
Henon map:} in the left frame we see the results of a phase space simulation and 
the initial small invariant circle. The right frame illustrates the results of the 
numerical continuation, making it clear that breakdown of the family 
involves loss of smoothness, as predicted by 
\cite{MR2672635}.  That is, the outermost invariant circles 
appear to have fairly sharp corners, indicating a blow up in the first 
Sobolev norm. 
} \label{fig:continuationPeriodOne}
\end{figure}

Consider for example a small circle, the orbit of $P = [0,0.1]$
in the area preserving Henon maps, and compute 
the parameterization and the rotation number 
$\rho \approx 0.550640092644521$.
%The method starts with 10 modes, increases the number of modes by 50 until either Newton-like operation fails to converge or we hit a pre-determined number of modes, in this case 510. 
We increment $\rho$ by $0.001$.  If the result converges we try again.  If not we increase the number 
of modes, and decrease the increment. 
% In 223 attempts, we had 33 successes, and
In this way we computed 150 invariant circles in the family, and 
 finish with a final increment on the order of $10^{-13}$.
 At this point the Sobolev norms are large, and we terminate the continuation.
The results are illustrated in Figure \ref{fig:continuationPeriodOne}.

%giving a final $\rho = 0.570614$, and locate $33$ successive nested invariant circles.
%, we have therefore a spread of approximately $0.0200$ on the rotation 
%number between the inner circle and the outer. 
%However, the Sobolev norms change dramatically.

%
%This large change in Sobolev norm would suggest a beginning in the break-down in regularity. 
%Another noteworthy feature of the norm surface is the ridges. 
%This behavior still indicates that the regularity does not decrease monotonically. 
%The algorithm only uses enough modes to decrease the tail size and error to a certain threshold (typically 1e-14). 
%There is usually some feature like "clumping" along the trajectories near these ridges where the behavior is more like a periodic trajectory. 
%In a perturbed integrable system the analogous situation is that an island chain is about to form near that orbit under a larger perturbation.  
%
%Looking to the log-plot of coefficients we see that the decay -- while rough --
%is actually quite rapid on  the last successful circle. 
%These features likely explain the dramatic increase in the high order Sobolev norms. 
%

\begin{figure}[t!]
\begin{centering}
\includegraphics[width=.45\textwidth]{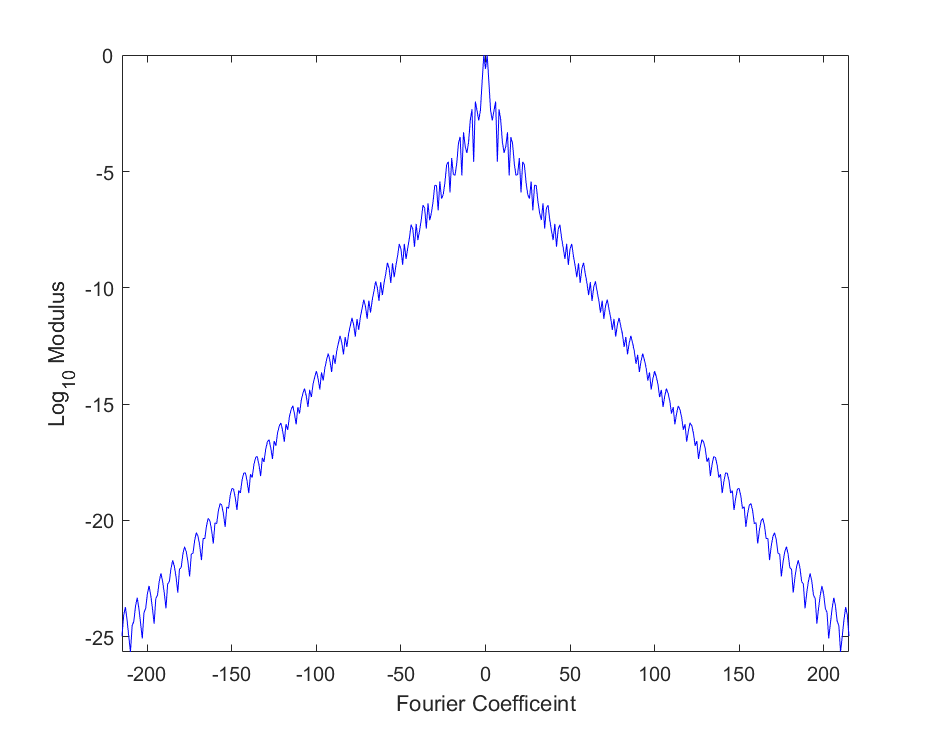}
\includegraphics[width=.45\textwidth]{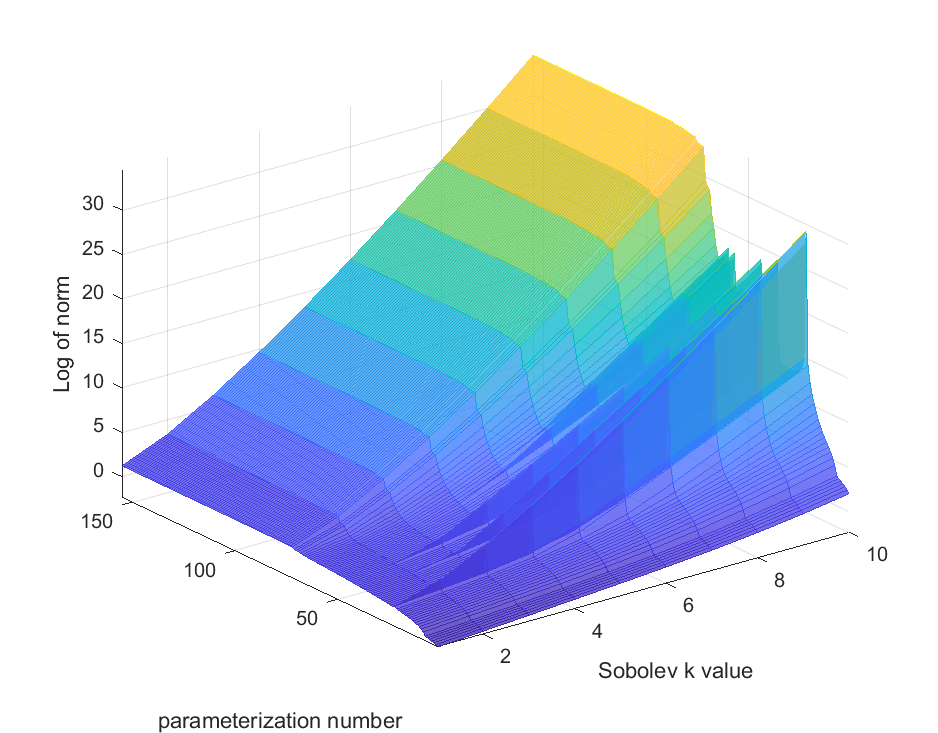}
\end{centering}
\caption{\textbf{Coefficient decay:} 
The left frame illustrates the decay of the coefficients 
for the last curve in our continuation computation, and we see
that even this circle has quite rapid decay.  However the right frame
illustrates the Sobolev norms calculated along the entire family, and 
we see that after about 50 continuation steps they are quite large.  
We note that for the remaining 100 steps, the step size is very small.
That is,  most of the progress is made in the first 50 steps.  
} \label{fig:decayNorms}
\end{figure}

\begin{figure}[t!]
\begin{centering}
\includegraphics[width=0.9\textwidth]{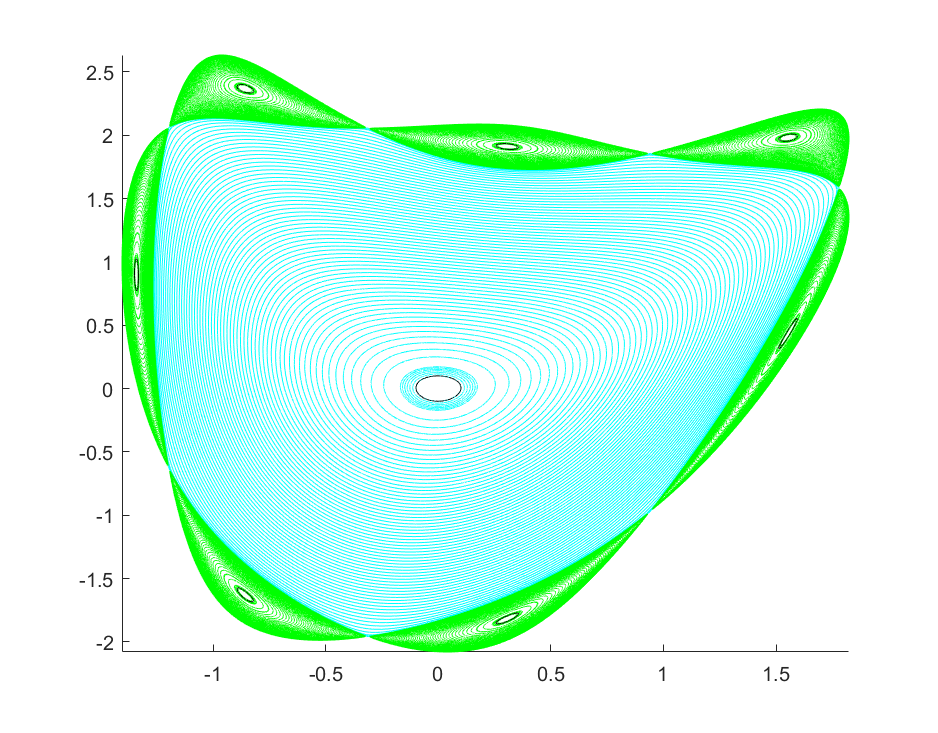}
\end{centering}
\caption{\textbf{Filling in the phase space:} after the period on family 
we restart the continuation on a period 7 family and continue until breakdown.
The union of the results yields a fairly large region of phase space covered
by the two Cantor sets.
} \label{fig:multipleFamilies}
\end{figure}

After the breakdown of the period one family we observe that there appears to be a period 
7 family of quasiperiodic circles.  We locate the period 7 orbit with a (finite dimensional)
Newton scheme, check the elliptic stability type, and look start our search for a period 
7 family nearby.  Once a single circle is found, we continue again until breakdown.
The result os a fairy large set of quasiperiodic motions.  
The results are illustrated in Figure \ref{fig:multipleFamilies}.
Continuing in this way, we find -- after the original period 1 and 
period 7 families -- a period 
1, 19, 1, 55, 1, 12, 120, and 17.
These results are illustrated in Figure \ref{fig:allFam}.

\begin{figure}[t!]
\begin{centering}
\includegraphics[width=.95\textwidth]{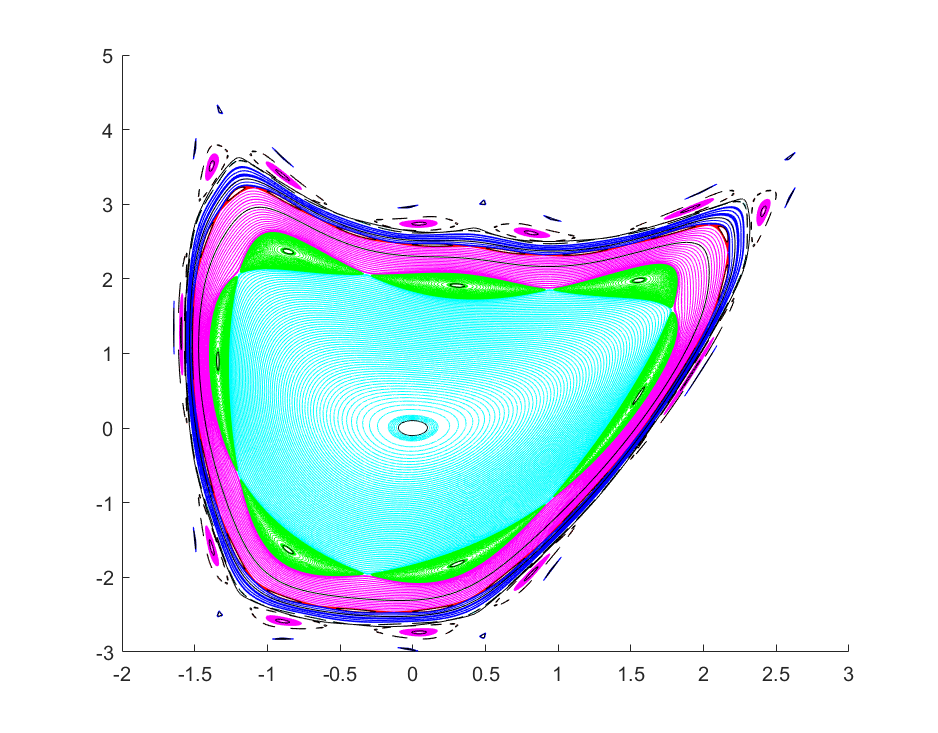}
\end{centering}
\caption{\textbf{Results of ten different continuations:} 
Note there are no phase space samples in this picture.  Only plots
of images of Fourier parameterizations.  In total --
starting from the origin and building out --
we see families of  
period 1 (teal), 7 (green), 1 (magenta), 19 (thin red band), 1 (blue), 
55 (thin bands just after blue), 1 (black), 12 (pink), 120 (black spots
around the pink family), and 17 (blue).  Some of the families are very 
thin.  We do not claim that there are no other tori in this region.   
} \label{fig:allFam}
\end{figure}

\section{Conclusions} 
\label{sec:conclusions}
The goal of this paper was to demonstrate that the method of weighted Birkhoff 
averages proves to be the perfect tool for initializing the parameterization 
method for invariant tori.  We have provided detailed example calculations 
illustrating the approach for classic polynomial and non-polynomial examples.
Moreover, we described and implemented a multiple shooting version of the 
parameterization method for simultaneously computing period-$d$ systems of 
invariant circles for $d$ as large as 120.  We also discuss a quadratic recast/automatic 
differentiation scheme which reduces the implementation of Newton scheme
to diagonal linear operators and discrete convolutions.  
We also introduced a global unfolding parameter for the parameterization method
which is built directly into the nonlinear conjugacy equation.
This avoids the need for introducing new parameters in the linear equations
at each step of the Newton method.  These ideas can be 
combined with basic numerical schemes to compute large sets of quasiperiodic 
motions.  Taken together, the approach described here provides a flexible
general toolkit for computing systems of invariant circles for area preserving maps.  

An natural future direction will be to extend the approach taken here for invariant 
2-tori in volume preserving maps.  For example combining the ergodic averages
for 2 tori used in \cite{MR4322369} with the parameterization method for 
volume preserving maps developed in \cite{MR3477312,MR3317251}.
We note for example that our unfolding parameter argument extends directly 
to this case. Extension to invariant tori in higher dimensional symplectic 
maps should be straight forward, but justifying the unfolding parameter
will require considering Calabi invariants.  The utility of Calabi invariants
in the parameterization method is discussed at length in 
\cite{MR4263036}.  Another valuable extension is to modify these ideas for 
application to parameterization of invariant tori for Hamiltonian ODEs,
as discussed in \cite{MR4361879,MR4209693}.
Indeed, the idea of combining rapidly converging Birkhoff averages 
with Newton schemes for solving invariance equations is so natural 
it is clear there will be many additional extensions and applications.

%
%\begin{itemize}
%\item Efficient algorithms to find rotation number and initial approximations
%\item Robust method that provides conditions to force solutions to be isolated in the function space.
%\item Diminishing returns on higher accuracy points to a possible optimum number of modes
%\item Could be a natural truncation point for CAP
%\item Lots of interesting interplay between the rotation numbers and the location of the orbits on the phase space, maybe worth looking into the direction of rotation of the tori
%\item Something about how the circles come in cantor-dust like groups so there is always a possibility of jumping over a feature during continuation.
%\item Maybe these are cantori that we are finding from Aubrey-Mather Theory, invariant cantor sets.
%\end{itemize}
%

\section{Acknowledgements} 
\label{sec:acknowledg}

The authors would like to thank Evelyn Sander
for illuminating conversations which inspired the present work.  
In particular, she explained the 
(then quite recent) results about weighted Birkhoff averages to the second Author during 
the 2014 AIMS Conference on Dynamical Systems and Applications in Madrid.  
The author's also thank 
Rafael de la Llave for a number of additional helpful discussions.
In particular, the idea of using the area preserving property of the dynamical system 
as a means to formulate an appropriate unfolding parameter for the parameterization 
method emerged during conversations with the second author during his 
visit to FAU in 2015.  Conversations with Alex Haro, and Jordi-Llu\'{i}s Figueras 
are also gratefully acknowledged.  
The work of the second author was partially supported by NSF grant  DMS-1813501
during some of the work on this project.

%
%
%\medskip
%\printbibliography

\bibliographystyle{alpha}
\bibliography{papers}

\end{document}